\documentclass{amsart}
\usepackage[all]{xy}
\usepackage[utf8x]{inputenc}
\usepackage{amssymb, amsmath,epsfig,url}
\usepackage[lmargin=1.1in,rmargin=1.1in,tmargin=1.1in,bmargin=1.1in]{geometry}
\newcommand{\color}[6]{}

\newtheorem{prop}{Proposition}[section]
\newtheorem{lem}[prop]{Lemma}
\newtheorem{defn}[prop]{Definition}

\newtheorem{thm}[prop]{Theorem}

\newtheorem{conj}[prop]{Conjecture}

\theoremstyle{definition}
\newtheorem{rem}[prop]{Remark}

\newcommand{\Comm}{\mathrm{C}}
\newcommand{\CComm}{\mathcal{C}}
\newcommand{\bdim}{\textbf{dim}}

\renewcommand{\AA}{\mathbb{A}}

\newcommand{\GG}{\mathbb{G}}
\newcommand{\LL}{\mathbb{L}}
\newcommand{\LLL}{\mathcal{L}}

\newcommand{\PP}{\mathbb{P}}

\renewcommand{\SS}{\mathfrak{M}}
\newcommand{\TT}{\mathcal{T}}
\newcommand{\VY}{\mathbb{A}}

\newcommand{\XX}{\mathcal{X}}
\newcommand{\YY}{\mathcal{Y}}
\newcommand{\muu}{\hat{\mu}}

\newcommand{\muhat}{\hat{\mu}}
\newcommand{\CC}{\mathcal{C}}
\newcommand{\DD}{\mathcal{D}}
\newcommand{\ZZ}{\mathbb{Z}}
\newcommand{\NN}{\mathbb{N}}

\newcommand{\nn}{\mathbf{n}}

\newcommand{\ehat}{\hat{e}}
\newcommand{\HH}{\mathbb{H}}

\newcommand{\OO}{\mathcal{O}}
\newcommand{\tw}{\mathtt{tw}}
\newcommand{\ob}{\mathtt{ob}}

\newcommand{\one}{\mathbf{1}}
\newcommand{\VV}{\mathcal{V}}
\newcommand{\sDet}{\mathtt{sDet}}
\newcommand{\TF}{\mathfrak{TW}}
\newcommand{\TN}{\mathfrak{nilp}}
\renewcommand{\TT}{\mathfrak{T}}
\newcommand{\EE}{\mathcal{E}}

\DeclareMathOperator{\Bkron}{\mathcal{B}}

\DeclareMathOperator{\rank}{rank}
\DeclareMathOperator{\KW}{con}
\DeclareMathOperator{\triang}{triang}
\DeclareMathOperator{\thick}{thick}
\DeclareMathOperator{\JH}{\mathtt{JH}}

\DeclareMathOperator{\KS}{KS}

\DeclareMathOperator{\rel}{rel}
\DeclareMathOperator{\op}{op}
\DeclareMathOperator{\Fun}{Fun}

\DeclareMathOperator{\Tot}{Tot}
\DeclareMathOperator{\virt}{virt}

\DeclareMathOperator{\kron}{Kron}
\DeclareMathOperator{\BBS}{\mathtt{BBS}}
\DeclareMathOperator{\edge}{E}
\DeclareMathOperator{\verts}{V}

\DeclareMathOperator{\BGL}{BGl}

\DeclareMathOperator{\sss}{ss}
\DeclareMathOperator{\con}{con}

\DeclareMathOperator{\SP}{sp}
\DeclareMathOperator{\Ker}{Ker}

\DeclareMathOperator{\End}{End}
\DeclareMathOperator{\rmod}{\mathtt{Mod-}}

\DeclareMathOperator{\Coh}{Coh}

\DeclareMathOperator{\DDb}{D^b}
\DeclareMathOperator{\KK}{K}
\DeclareMathOperator{\Var}{Var}
\DeclareMathOperator{\Sch}{Sch}

\DeclareMathOperator{\Sta}{St^{aff}}

\DeclareMathOperator{\Gl}{Gl}

\DeclareMathOperator{\Spec}{Spec}
\DeclareMathOperator{\Sym}{Sym}
\DeclareMathOperator{\pt}{pt}

\DeclareMathOperator{\Ho}{H}
\DeclareMathOperator{\Der}{D^b}
\DeclareMathOperator{\Derbe}{D^b_{exc}}
\DeclareMathOperator{\Derf}{D_{f.d}}
\DeclareMathOperator{\Dercpct}{D_{cpct}}
\DeclareMathOperator{\Dern}{D_{nilp}}

\DeclareMathOperator{\DER}{D}

\DeclareMathOperator{\Vect}{Vect}
\DeclareMathOperator{\Lo}{\mathcal{L}}

\DeclareMathOperator{\Hom}{Hom}
\DeclareMathOperator{\RHom}{RHom}
\DeclareMathOperator{\Ext}{Ext}
\DeclareMathOperator{\Mat}{Mat}

\DeclareMathOperator{\tr}{tr}

\DeclareMathOperator{\Cp}{\mathbb{C}}

\DeclareMathOperator{\nilp}{nilp}

\DeclareMathOperator{\Image}{Image}

\DeclareMathOperator{\mm}{\mathbf{m}}

\DeclareMathOperator{\id}{id}

\DeclareMathOperator{\constr}{constr}

\title{The motivic Donaldson--Thomas invariants of (-2)-curves}
\author{Ben Davison, Sven Meinhardt}

\begin{document}

\maketitle

\begin{abstract}
In this paper we calculate the motivic Donaldson--Thomas invariants for (-2)-curves arising from 3-fold flopping contractions in the minimal model programme.  We translate this geometric situation into the machinery developed by Kontsevich and Soibelman \cite{KS}, and using the results and framework of \cite{DM11a} we describe the monodromy on these invariants.  In particular, in contrast to all existing known Donaldson--Thomas invariants for small resolutions of Gorenstein singularities these monodromy actions are nontrivial.
\end{abstract}

\tableofcontents

\section{Introduction}
Motivic Donaldson--Thomas invariants were introduced in \cite{KS}, as a generalisation of the classical theory of Donaldson--Thomas invariants initiated in \cite{casson}.  At the same time Joyce \cite{Jo06a,Jo06b,Jo07a,Jo07c,JoyceHGF,Jo07b,Jo08a} and Joyce and Song \cite{JS08} rigorously extended classical Donaldson--Thomas theory to take care of the technicalities involved in dealing with strictly semistable coherent sheaves on Calabi--Yau 3-folds, and in this framework formulated a deep integrality conjecture regarding the resulting Donaldson--Thomas invariants.  Assuming the more ambitious framework of \cite{KS}, integrality properties of generalised Donaldson--Thomas invariants are conjecturally obtained by taking the Euler characteristic of motivic Donaldson--Thomas invariants, after multiplication by the motive $\Cp^*$; such statements are supposed to be a shadow of the fact that these invariants, which are a priori only stack valued, are in fact variety valued, so that taking Euler characteristic is legitimate, and produces integers.
\medbreak
If $Y\rightarrow X$ is a small resolution of a toric Gorenstein singularity, the calculation of the motivic Donaldson--Thomas invariants of $Y$ has by now received a fairly comprehensive treatment (see \cite{BBS,MMNS11,MN11}).  Let $\KK^{\muhat}(\Var/\Spec(\Cp))$ be the ring of $\muhat$-equivariant varieties, then there is a ring homomorphism \\$\KK^{\muhat}(\Var/\Spec(\Cp))[\LL^{1/2}]\rightarrow \ZZ[q^{1/2}]$, obtained by first taking the Hodge spectrum, a homomorphism to the ring of polynomials in fractional powers of two variables $u$ and $v$, and then specialising $u=v=q^{1/2}$.  Furthermore, this is a retraction of rings, since there is a right inverse taking $q^{1/2}$ to $-\LL^{1/2}$.  The Donaldson--Thomas invariants that arise in the study of the above toric resolutions $Y\rightarrow X$ all lie in the obviously very well-understood subring that is the image of this retract.
\medbreak
By contrast, the ring $\KK^{\muhat}(\Var/\Spec(\Cp))[\LL^{1/2}]$, as a whole, has a rich ring structure, with the product given by Looijenga's ``exotic'' convolution product (see \cite{Looi00,GLM06,KS}), and a pre $\lambda$-ring structure utilised in \cite{DM11a} to express the motivic DT invarants of the one loop quiver with potential --- this was the first case to really make use of this extra structure.
\medbreak
The present paper represents perhaps the first case where ``natural'' Donaldson--Thomas invariants living in the interesting part of the ring $\KK^{\muhat}(\Var/\Spec(\Cp))[\LL^{1/2}]$ are discussed.  Of course the question of naturalness here is subjective --- we are appealing to the sensibilities of algebraic geometers, in that we consider an example that is manifestly a part of 3-dimensional geometry, as opposed to the case of the one loop quiver with potential, which in the homogeneous case gives rise to the algebra $\Cp[x]/(x^d)$, which looks rather more like zero-dimensional geometry.  More specifically, we consider the motivic Donaldson--Thomas invariants of (-{2})-curves, which are, for us, resolutions $Y_d\rightarrow X_d$ of singularities as defined in Equation (\ref{Xddef}).  In birational geometry and physics, these curves have a very long and rich history, see \cite{Rei83,Lau81,Kol89,KaMo92} for example.
\medbreak
Our paper also seems to represent the first serious attempt to calculate Donaldson--Thomas invariants while keeping as true as possible to the framework of \cite{KS}.  A side-effect of this approach is that some discussion of orientation data is necessitated.  It is hoped that seeing this aspect of the story in action will help to demystify it a little.  For the sake of those who would like to swap the (ever decreasingly) conjectural framework of \cite{KS} for the single very reasonable-looking conjecture of \cite{DM11a}, we prove a slight variant of our main result at the end of the paper, avoiding all mention of orientation data, cyclic 3-Calabi--Yau categories and minimal potentials.  
\medbreak
In both cases we work with an algebraic model of the derived category of compactly supported coherent sheaves on $Y_d$, provided by considering modules over an algebra $A_{Q_{-2},W_d}$, which is the free path algebra of the quiver in Figure \ref{NCres}, quotiented by some relations determined by the noncommutative derivatives of a potential $W_d$.  Our main result is Theorem \ref{mainthm}, which states that
\begin{equation}
\label{fivefour}
\Phi_{Q_{-2},W_d}\left([\XX^{\nilp}_{Q_{-2},W_d}]\right)=\Sym\left(\sum_{n\geq 0}\frac{\LL^{-1/2}(1-[\mu_{d+1}])}{\LL^{1/2}-\LL^{-1/2}}\left(\ehat_{(n,n+1)}+\ehat_{(n+1,n)}\right)+\sum_{n\geq 1}\frac{\LL^{-1/2}+\LL^{-3/2}}{\LL^{1/2}-\LL^{-1/2}}\ehat_{(n,n)}\right),
\end{equation}
where the quantity on the left hand side is by definition the motivic generating series for nilpotent modules over $A_{Q_{-2},W_d}$, which on the geometric side of Van den Bergh's equivalence corresponds to counting coherent sheaves on the exceptional locus of $Y_d\rightarrow X_d$.  In more detail, the variables $\ehat_{(n,m)}$ keep track of the Chern classes of the sheaves we are counting, under the transformation
\[
(n,m)\mapsto (n-m)[C_d]+m[\pt],
\]
where $C_d$ is the exceptional curve of the resolution $Y_d\rightarrow X_d$.  Equation (\ref{fivefour}) implies, by the definition of motivic Donaldson--Thomas invariants $\Omega^{\nilp}$, that they are given by 
\[
\Omega^{\nilp}(\nn)=\begin{cases} \LL^{-1/2}(1-[\mu_{d+1}])& \textrm{if }\nn=(n,n+1)\textrm{ or }\nn=(n+1,n)\\ \PP^1\cdot \LL^{-3/2}&\textrm{if }\nn=(n,n).\end{cases}
\]
Here $\mu_{d+1}$ is considered as a $\mu_{d+1}$-equivariant variety in the natural way, and so we have indeed produced motivic Donaldson--Thomas invariants with nontrivial monodromy, arising ``in nature'' e.g. string theory, and confirmed integrality, all the way up to the motivic level, for the Donaldson--Thomas invariants of (-2)-curves.
\medbreak
The structure of the paper is as follows.  In Section \ref{gengeom} we collect together facts regarding the algebraic geometry and noncommutative algebraic geometry of (-2)-curves, in particular introducing the explicit noncommutative algebra $A_{Q_{-2},W_d}$ whose noncommutative Donaldson--Thomas theory we subsequently study.  This version of Donaldson--Thomas theory is motivic; in Section \ref{motsec} we explain what the word ``motivic'' means, by introducing all the relevant technicalities on motivic vanishing cycles, motivic Hall algebras and pre $\lambda$-ring structures on ``naive'' Grothendieck rings of motives.  These are the rings in which motivic DT invariants live.  In Section \ref{motdt} we explain how these invariants are defined; we introduce requisite definitions and facts regarding 3--Calabi--Yau categories and orientation data.  Orientation data is a concept introduced in \cite{KS}, and is an extra structure that one must put on a 3--Calabi--Yau category in order to be able to define motivic DT invariants for that category.  Furthermore, the motivic DT invariants will in general depend on the choice that we make.  We recall how to control this choice with Proposition \ref{comparisonthm}, which states that for the natural choice of orientation data provided by the presentation of a Jacobi algebra as the algebra arising from a quiver with potential, the integration map agrees with an analogue of the integration map considered by Behrend, Bryan and Szendr\H{o}i in \cite{BBS}.  Finally, in Section \ref{calcs} we present our results.  To start with, we work within the framework of motivic Donaldson--Thomas theory established by Kontsevich and Soibelman in \cite{KS} to prove Theorem \ref{mainthm}, which is Equation (\ref{fivefour}), and concerns the Donaldson--Thomas theory of sheaves on $Y_d$ supported on the exceptional locus --- in particular we calculate the contribution of the exceptional curve itself.  Secondly, we present a calculation of the motivic DT invariants of the category of compactly supported sheaves on the whole of $Y_d$, working with the somewhat more down-to-earth integration map of Behrend, Bryan and Szendr\H{o}i, and a conjectural identity regarding motivic vanishing cycles.
\section{The geometry of (-2)-curves}
\label{gengeom}
In this paper we study the motivic Donaldson--Thomas invariants of local (-2)-curves, which are defined in the following way, following \cite[Sec.5]{Rei83}.  We assume that $f\colon Y\rightarrow X$ is a resolution of a Gorenstein complex 3-fold singularity with exceptional curve $C\cong\PP^1$, satisfying the condition that $f^*\omega_X\cong\omega_{Y}$, $\omega_Y\cdot C=0$ and $N_{C|Y}\cong \OO_{C}\oplus \OO_{C}(-2)$ or $N_{C|Y}\cong \OO_{C}(-1)\oplus \OO_{C}(-1)$.  Then (see \cite[(5.13)]{Rei83} and the surrounding discussion) we may assume that $X$ is one of the singularities
\begin{equation}
\label{Xddef}
X_d=\Spec \Bigl(\Cp[x,y,z,w]/\bigl(x^2+y^2+(z+w^d)(z-w^d)\bigr) \Bigr)
\end{equation}
for $d\geq 1$, and $Y$ is given by one of the two resolutions provided by blowing up along $0=x=z\pm w^d$.  We denote by $Y_d$ the blowup along $0=x=z+w^d$, and by $Y_d^+$ the blowup along $0=x=z-w^d$.  We will refer to the exceptional rational curve in $Y_d$ always as $C_d$, to make it clear which resolution of singularities it belongs to.  The birational morphism $Y_d\dashrightarrow Y_d^+$ is the flop of the curve $C_d$, and there is an equivalence of categories 
\begin{equation}
\label{flopDE}
\DDb(\Coh(Y_d))\rightarrow \DDb(\Coh(Y_d^+))
\end{equation}
with Fourier--Mukai kernel $\OO_{Y_d\times_{X_d}Y^+_d}$.  This is an example of a generalised spherical twist (see \cite{Tod07}).  This equivalence is not given by an equivalence of the hearts of these two categories (even though they are in fact equivalent, as there is an obvious isomorphism of schemes $Y_d\rightarrow Y_d^+$).  As in \cite[(5.3)]{Rei83} one defines the width\footnote{Not to be confused with the \textit{length} of $C_d$, which is an entirely different invariant introduced by Koll\'ar in \cite{CKM88} and used in the classification by S. Katz and D. Morrison \cite{KaMo92} of irreducible small resolutions of Gorenstein 3-fold singularities.} of $C_d$ to be the length of the component of the moduli space of coherent sheaves on $Y_d$ containing $\OO_{C_d}$.  One can show from the explicit description of $X_d$ and $Y_d$ that the width of $C_d\subset Y_d$ is $d$. \medbreak
For the purposes of this paper we will be interested in a derived equivalence that is different to that of Equation (\ref{flopDE}). That is, we will be interested in a derived equivalence between the category of coherent sheaves on $Y_d$ and the category of finitely generated right modules $\rmod A_{Q_{-2},W_d}$ for a noncommutative algebra $A_{Q_{-2},W_d}$.  The approach to defining and studying Donaldson--Thomas invariants of categories of coherent sheaves is as initiated by Szendr\H{o}i in \cite{conifold}, where the case of the ``noncommutative conifold'' is considered, and indeed we will recover (motivic) Donaldson--Thomas invariants for the noncommutative conifold, as it is a (-2)-curve of width 1.\footnote{Note that the motivic Donaldson--Thomas invariants we obtain for the conifold differ from those of \cite{conifold,MMNS11}; this is a result of a different choice of orientation data, in the terminology of \cite{KS}.  We revisit this subtle point in Remark \ref{conremark}.}\medbreak
The existence of the algebra $A_{Q_{-2},W_d}$ satisfying 
\begin{equation}
\label{ncDE}
\xymatrix{
\DDb(\rmod A_{Q_{-2},W_d})\ar[r]^-{\sim}&\DDb(\Coh(Y_d))
}
\end{equation}
is provided by the results of Van den Bergh \cite[Thm.5.1]{NonComCrep}.  It will help to have an explicit description of $Y_d$.  It is covered by two coordinate patches 
\begin{align*}
U_1=&\Spec(\Cp[x,y_1,y_2])\\
U_2=&\Spec(\Cp[w,z_1,z_2]), 
\end{align*}
which are glued along
\begin{align*}
x=&w^{-1}\\
z_1=&x^2y_1+xy_2^d\\
z_2=&y_2.
\end{align*}
In the case of the conifold, after the change of coordinates $z'_1=wz_1-z_2$, $z_2'=-(1+w)z_1+z_2$, $y'_1=y_1$, $y'_2=(x+1)y_1+y_2$, we recover the usual presentation of the resolved conifold as the total space of the bundle $\OO_{C_1}(-1)\oplus\OO_{C_1}(-1)$ over $C_1\cong\PP^1$.  We define $\OO_{Y_d}(-n):=\OO_{Y_d}(nD)$ where $D$ is the divisor cut out by the equation $x=0$ in the above coordinate patches.  Then by Van den Bergh's theorem, we have a derived equivalence as in (\ref{ncDE}) if we set 
\begin{equation}
\label{Appear}
A_{Q_{-2},W_d}:=\End_{Y_d}(E_d), 
\end{equation}
where we define \[
E_d:=\OO_{Y_d}\oplus\OO_{Y_d}(-1).
\]
We follow the convention of \cite{KaMo06}, representing morphisms between the two line bundles $\OO_{Y_d}$ and $\OO_{Y_d}(-1)$ by elements of $\Cp[w,z_1,z_2]$ under the identifications $\Gamma(U_2,\OO_{Y_d})\cong\Cp[w,z_1,z_2]\cong \Gamma(U_2,\OO_{Y_d}(-1))$.  The endomorphism algebra can then be represented by the quiver algebra depicted in Figure \ref{NCres}.  We have the relations
\begin{align}
\label{relations}
AX=&YA\\
\nonumber BX=&YB\\
\nonumber XC=&CY\\
\nonumber XD=&DY\\
\nonumber X^d=&CA-DB\\
\nonumber Y^d=&AC-BD.
\end{align}
It follows that $A_{Q_{-2},W_d}$ admits a superpotential description in the sense of \cite{ginz}, with quiver $Q_{-2}$ given by the quiver of Figure \ref{NCres} and superpotential given by 
\begin{equation}
W_d=\frac{1}{d+1}X^{d+1}-\frac{1}{d+1}Y^{d+1}-XCA+XDB+YAC-YBD.
\end{equation}
That is, we have an isomorphism
\begin{equation}
\label{jacdefe}
A_{Q_{-2},W_d}\cong \mathbb{C} Q_{-2}/\langle \partial W_d/\partial E,E\in \edge(Q_{-2})\rangle,
\end{equation}
where for a general quiver $Q$ and $W\in \mathbb{C} Q/[\mathbb{C} Q, \mathbb{C} Q]$ given by a single cycle, and $E\in \edge(Q)$ an arrow,
\begin{equation}
\label{ncder}
\partial W/\partial E:=\sum_{aEb=W ,\; a\mbox{ \scriptsize and } b \mbox{ \scriptsize  paths in } Q}ba,
\end{equation}
and for general $W$, $\partial W/\partial E$ is defined by extending linearly.
\begin{defn}
\label{jacdef}
For a general quiver $Q$ with potential $W$, we define $A_{Q,W}$ in the same way as in Equation (\ref{jacdefe}).  This is called the Jacobi algebra associated to the pair $(Q,W)$.
\end{defn}

\medbreak
\begin{figure}
\centering
\input{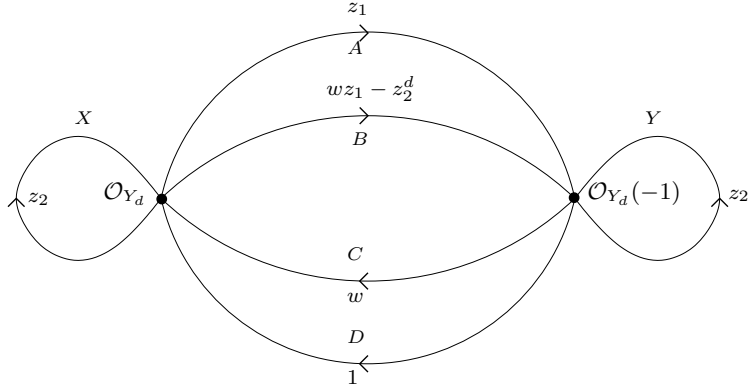}
\caption{The quiver $Q_{-2}$.  The vertices are marked with the summands of the bundle $E_d$, and the arrows are marked with morphisms between these summands.}
\label{NCres}
\end{figure}

\begin{rem}
\label{conifold}
In the case of the conifold (i.e.\ if $d=1$) there is a simpler Jacobi algebra presentation of the noncommutative resolution $\End_{Y_d}(E_d)$, considered in \cite{conifold}, see also Remark \ref{conremark}.  The quiver is given by $Q_{\mathrm{con}}$, which is $Q_{-2}$ with the two loops $X$ and $Y$ removed --- see Figure \ref{KWquiver}.  One sets $W_{\mathrm{con}}=ACBD-ADBC$, and one can show directly that $A_{Q_{\mathrm{con}},W_{\mathrm{con}}}\cong A_{Q_{-2},W_1}$.  Note that the relations (\ref{relations}) imply the relations given by the noncommutative derivatives of $W_{\con}$, considered as a superpotential for $Q_{-2}$.  As a result one may consider the morphisms assigned to $X$ and $Y$ for a $A_{Q_{-2},W_d}$-module $M$ as being together an endomorphism of a module $M_{\con}$ for $A_{Q_{\con},W_{\con}}$, where $M_{\con}$ in turn is determined by the morphisms assigned to $A$, $B$, $C$ and $D$ by $M$, via the forgetful map.
\end{rem}
\section{Naive Grothendieck rings of motives}
\label{motsec}
\subsection{A pre $\lambda$-ring of motives}
\label{lambdarings}
In this section we recall the construction of ``naive'' Grothendieck pre $\lambda$-rings of $\muhat$-equivariant motives, or motives carrying a monodromy action.  The reason for introducing such rings is that they are the natural home of motivic vanishing cycles, which carry monodromy actions in analogy with their sheaf-theoretic cousins.  The reason for taking special care of the monodromy is that while in general the map induced on naive Grothendieck rings of motives by forgetting the monodromy will be a homomorphism of underlying groups, it will fail to respect the multiplication or pre $\lambda$-ring operations.  In particular, both the ``integration map'' (\ref{intdef}) of Kontsevich and Soibelman and the map (\ref{BBSmap}) generalising the map exploited by Behrend, Bryan and Szendr\H{o}i in \cite{BBS} will fail to be algebra homomorphisms for general quivers with potential if we forget monodromy.  
\medbreak
For $\SS$ an Artin stack locally of finite type over $\Cp$ we define $\KK_0(\Sta/\SS)$ to be the Abelian group which is generated by isomorphism classes of morphisms $X\xrightarrow{f} \SS$ of finite type, with $X$ a separated reduced stack over $\Cp$ satisfying the condition that each of its $\Cp$-points has affine stabiliser, subject to the relations 
\[ 
[X\xrightarrow{f} \SS ] \sim [Z\xrightarrow{f|_Z} \SS ] + [X\setminus Z \xrightarrow{f|_{X\setminus Z}} \SS], 
\] 
for $Z\subset X$ a closed substack of $X$.  If 
\[
(\SS,\epsilon\colon \SS\times\SS\rightarrow \SS,0\colon\Spec(\Cp)\rightarrow \SS)
\]
is a (commutative) monoid in the category of Artin stacks over $\Cp$ with $\epsilon$ of finite type, then $\KK_0(\Sta/\SS)$ acquires the structure of a (commutative) $\KK_0(\Sta/\Spec(\Cp))$-algebra, via convolution and the inclusion 
\[
[X\xrightarrow{f} \Spec(\Cp)]\mapsto [X\xrightarrow{0\circ f} \SS].
\]
There are obvious $G$-equivariant versions $\KK_0^G(\Sta/\SS)$ of the above groups and rings for $G$-equivariant stacks or monoids $\SS$, where we work with $G$-equivariant morphisms and assume that every point in $X$ lies in a $G$-equivariant affine neighbourhood.  Again we consider $\KK^G_0(\Sta/\SS)$ as a $\KK_0(\Sta/\Spec(\Cp))$-algebra if $\SS$ is a monoid in the category of locally finite type $G$-equivariant Artin stacks with finite type monoid map.  For technical reasons it is better to make the following modifications to $\KK^G_0(\Sta/\SS)$, forming the modified ring $\KK^G(\Sta/\SS)$:
\begin{enumerate}
\item
\label{rel1}
If $X'\xrightarrow{\pi} X$ is a $G$-equivariant vector bundle of rank $r$ then we impose the relation 
\[
[X'\xrightarrow{f\circ \pi}\SS]\sim \LL^r\cdot [X\xrightarrow{f}\SS]
\]
in $\KK^G(\Sta/\SS)$, where $\LL$ is the class of the affine line $\mathbb{A}_{\Cp}^1$ in $\KK(\Sta/\Spec(\Cp))$. 
\item
\label{con2}
In addition, we complete with respect to the topology having as closed neighbourhoods of zero the subgroups 
\[
\KK_{\mathfrak{U}}:=\{L\in \KK^G_0(\Sta/\SS)\textrm{ such that }L|_{\mathfrak{U}}=0\}
\]
for $\mathfrak{U}\subset \SS$ an open substack.  In the sequel we always complete with respect to the analogous system of neighbourhoods, so for example the statement of Proposition \ref{localis} concerns expressions with infinitely many denominators $[\Gl_{\Cp}(n)]^{-1}$ if the stack $\SS$ is not of finite type.  If the base stack $\SS$ is of finite type this second modification makes no difference.
\end{enumerate}
\medbreak
We define in the natural way the subgroup (or subring, if $\SS$ is a monoid) $\KK^G(\Var/\SS)$, spanned by classes $[X\rightarrow \SS]$ for $X$ a $G$-equivariant variety over $\Cp$.
\medbreak
By \cite[Prop.1.1]{Eke09} there is an equality $[\Gl_{\Cp}(n)]=\prod_{0\leq i \leq n-1}(\LL^n-\LL^i)$ in $\KK(\Var/\Spec(\Cp))$.

\begin{prop}[cf.\ \cite{Eke09}, Theorem 1.2]
\label{localis}
The natural map 
\[
\Psi\colon\KK^G(\Var/\SS)[[\Gl_{\Cp}(n)]^{-1},n\in\NN]\rightarrow \KK^G(\Sta/\SS)
\]
is an isomorphism.
\end{prop}
For a morphism $h\colon \SS\rightarrow \TT$ of locally finite type Artin stacks we define 
\[
h^*\colon\KK^G(\Sta/\TT)\rightarrow \KK^G(\Sta/\SS)
\]
via the pullback.  If $h$ is representable there is an equality 
\[
h^*=\Psi\circ(h^*)|_{\Var}\circ\Psi^{-1}
\]
where $(h^*)|_{\Var}$ is the $\KK(\Var/\Spec(\Cp))[[\Gl_{\Cp}(n)]^{-1},n\in\NN]$-linear extension of the restriction of $h^*$ to a map 
\[
\KK^G(\Var/\TT)\rightarrow \KK^G(\Var/\SS).
\]
We define 
\[
\int_h\colon \KK^G(\Var/\SS)\rightarrow \KK^G(\Var/\TT)
\]
via composition with $h$, if $h$ is of finite type.  For $j\colon \SS'\hookrightarrow \SS$ an inclusion of a finite type substack we write $\int_{\SS'}:=\int_{h}\circ j^*$, where $h\colon\SS'\rightarrow \Spec(\Cp)$ is the structure morphism.
\medbreak
We will briefly recall the framework of \cite{DM11a}.  Let $(\SS,\epsilon,0)$ be a monoid in the category of Artin stacks, locally of finite type, with $\epsilon$ of finite type.  We wish to define a ``naive'' Grothendieck pre $\lambda$-ring of motives over $\SS$, where such motives are to carry a monodromy action.  When it comes to defining the pre $\lambda$-ring operations, it turns out to be most instructive to consider such motives via their associated mapping tori.  For this reason, we will be interested in the group $\KK^{\GG_m,n}(\Sta/\AA_{\SS}^1)$, the naive Grothendieck group of $\GG_m$-equivariant stacks over 
\[
\AA_{\SS}^1:=\AA_{\Cp}^1\times \SS.
\]
The stack $\AA_{\SS}^1$ is given the $\GG_m$-action that is trivial on $\SS$ and acts with weight $n$ on $\AA_{\Cp}^1$.  The  $\GG_m$-equivariant projection map 
\[
p\colon\AA_{\Cp}^1\times \SS\rightarrow \SS
\]
induces a map 
\[
p^*\colon\KK^{\GG_m}(\Sta/\SS)\rightarrow \KK^{\GG_m,n}(\Sta/\AA_{\SS}^1)
\]
and we denote by $\mathfrak{I}_n$ the image of this map.  We give $\SS$ the trivial $\mu_n$-action, where $\mu_n$ denotes the group of $n$-th roots of unity in $\mathbb{C}^*$.  The map $\KK^{\mu_n}(\Sta/\SS)\rightarrow \KK^{\GG_m,n}(\Sta/\AA_{\SS}^1)/\mathfrak{I}_n$ given by 
\[
[Y\xrightarrow{f} \SS]\mapsto [Y\times_{\mu_n} \GG_m\xrightarrow{(y,z)\mapsto (z^n,f(y))}  \AA_{\Cp}^1\times \SS]
\]
is an isomorphism.  For each $a\geq 1$ there is a natural morphism $\mu_{an}\rightarrow \mu_n$ sending $z$ to $z^a$, and this induces an inclusion $\KK^{\mu_n}(\Sta/ \SS)\rightarrow \KK^{\mu_{an}}(\Sta/\SS)$, and we define $\KK^{\muu}(\Sta/\SS)[\LL^{-1/2}]$ to be the group obtained by taking the direct limit of these inclusions, and then adding a formal square root to the inverse of $\LL$.  
\begin{defn}
Given a ring $R$, always assumed to be commutative, a pre $\lambda$-ring structure on $R$ is given by a map $\sigma\colon R\rightarrow R[[T]]$ staisfying
\begin{itemize}
\item
$\sigma(0)=1$
\item
$\sigma(a)=1+aT \mbox{ modulo }T^2\cdot R[[T]]$
\item
$\sigma(a+b)=\sigma(a)\sigma(b)$.
\end{itemize}
One defines the operations $\sigma^n(r)$ via $\sigma(r)=\sum_{i\geq 0} \sigma^i(r)T^i$, and we define $\Sym(r)=\sum_{i\geq 0}\sigma^i(r)$ when this infinite sum exists.\footnote{This will be the case for $r\in F^1R$ if $R$ is a complete filtered ring with filtration $F^\ast R$ such that $\sigma^i(F^j R)\subset F^{i\cdot j}R$ for all $i,j\in \NN$.}  Finally, if $R$ is a pre $\lambda$-ring we define a pre $\lambda$-ring structure on $R[[X]]$ by setting $\sigma^n(r\cdot X^i):=\sigma^n(r)\cdot X^{i\cdot n}$, extending to polynomials in $X$ by the equation $\sigma(a+b)=\sigma(a)\sigma(b)$, and completing with respect to the ideal generated by $X$.
\end{defn}
We always assume that $\sigma(1)=(1-T)^{-1}$ --- in other words we always pick $1$ to be a line element.  Using the above notation, we have $\Sym(\sum_{i\geq 1}a_i\cdot X^i)=\sum_{\pi}\prod_i(\sigma^{\pi(i)}(a_i))X^{i\pi(i )}$, where the sum is over all partitions $\pi$, and we denote by $\pi(i)$ the number of parts of $\pi$ of size $i$.
\begin{prop}\cite[Lem.4.1]{DM11a}
\label{lambdamot}
Let $(\SS,\epsilon,0)$ be a commutative monoid in the category of locally finite type schemes over $\Cp$ with $\epsilon$ of finite type.  Consider the map 
\begin{align*}
+\colon &\AA_{\SS}^1\times \AA_{\SS}^1\rightarrow \AA_{\SS}^1
\\
&((z_1,x_1),(z_2,x_2))\mapsto (z_1+z_2,\epsilon(x_1,x_2))
\end{align*}
making $\AA_{\SS}^1$ into a commutative monoid.  The Abelian group $\KK^{\GG_m,n}(\Var/\AA^1_{\SS})$ has the structure of a pre $\lambda$-ring, if we set 
\[
[X_1\xrightarrow{f_1}\AA^1_{\SS}]\cdot [X_2\xrightarrow{f_2}\AA^1_{\SS}]=[X_1\times X_2\xrightarrow{f_1\times f_2}\AA_{\SS}^1\times\AA_{\SS}^1 \xrightarrow{+}\AA_{\SS}^1]
\]
and
\[
\sigma^n([X\xrightarrow{f} \AA^1_{\SS}])=[\Sym^n X\xrightarrow{\Sym^n f} \Sym^n \AA_{\SS}^1\xrightarrow{+}\AA_{\SS}^1],
\]
for varieties $X,X_1,X_2$.  Furthermore, this induces a pre $\lambda$-ring structure on the quotient $\KK^{\mu_n}(\Var/\SS)$, which is preserved by the embeddings $\KK^{\mu_{n}}(\Var/\SS)\rightarrow \KK^{\mu_{an}}(\Var/\SS)$, giving rise to a pre $\lambda$-ring structure on $\KK^{\muu}(\Var/\SS)$.
\end{prop}
\begin{rem}
\label{stinclusion}
There is a unique extension of this pre $\lambda$-ring structure to $\KK^{\muu}(\Sta/\SS)[\LL^{-1/2}]$, using Proposition \ref{localis}.  This we may describe as follows.  Given an element in 
\[
\KK^{\muu}(\Var/\SS)[\LL^{-1/2},[\Gl_{\Cp}(n)]^{-1} , n\in\NN]
\]
one obtains a non-unique element $P$ of $\KK^{\muu}(\Var/\SS)[[u]][u^{-1}]$ after substituting instances of $[1-\LL^n]^{-t}$ out for their power series expansions and then substituting\footnote{There is a potentially confusing choice of sign here, especially since either choice of sign gives a formal square root of $\LL$.  We justify our choice by noting that $\LL$ is supposed to be the motive of $\Ho_c(\mathbb{A}^1,\mathbb{Q})$, the tensor square root of which has odd cohomological degree.} $\LL^{1/2}\mapsto -u$.  It is not hard to verify that the set of formal power series obtained in this way is closed under taking $\sigma^i$ for each $i$ if we extend $\sigma^i$ to the ring of Laurent series by $\sigma^i(au^j)=\sigma^i(a)u^{i\cdot j}$.  One may then take $\sigma^i(P)$, followed by the substitution $u\mapsto -\LL^{1/2}$, to arrive at a (unique) element of $\KK^{\muu}(\Var/\SS)[\LL^{-1/2},[\Gl_{\Cp}(n)]^{-1},n\in\NN]$.  See \cite{DM11a} for details.
\end{rem}
\begin{defn}
A power structure on a ring $R$ is a map $\left(1+T\cdot R[[T]]\right)\times R\rightarrow \left(1+T\cdot R[[T]]\right)$, written $(A(T),m)\mapsto A(T)^m$, satisfying
\begin{itemize}
\item
$A(T)^0=1$, 
\item
$A(T)^1=A(T)$, 
\item
$(A(T)\cdot B(T))^m=A(T)^m\cdot B(T)^m$, 
\item
$A(T)^{m+n}=A(T)^m\cdot A(T)^n$, $A(T)^{m\cdot n}=(A(T)^m)^n$, 
\item
$(1+T)^m$ is equal to $1+m\cdot T$ modulo $T^2\cdot R[[T]]$, 
\item
$A(T^a)^m=A(T)^m_{T\mapsto T^a}$.  
\end{itemize}
We assume all power structures are continuous with respect to the $T$-adic topology on $R[[T]]$.
\end{defn}
Given a power series $A(T)\in 1+T\cdot R[[T]]$ with $R$ a pre $\lambda$-ring, we may write $A(T)$ uniquely as an expression
\begin{equation}
\label{AExp}
A(T)=\Sym\left(\sum_{n\geq 1}a_nT^n\right).
\end{equation}
It follows that there is a one to one correspondence between continuous power structures and pre $\lambda$-ring structures: given a pre $\lambda$-ring structure we write $A(T)^m=\Sym(\sum_{n\geq 1}ma_n T^n)$, with $a_n$ defined by (\ref{AExp}), and given a power structure on $R$ we may write $\sigma(m)=(1-T)^{-m}$.  For $R$ a ring, and $R'$ a quotient ring of $R$, power structures on $R$ descending to power structures on $R'$ are exactly the power structures such that the associated pre $\lambda$-ring structure descends to $R'$.
\begin{prop}
\label{powerstruc}
The power structure on $\KK^{\GG_m,n}(\Var/\AA_\SS^1)$ inducing the pre $\lambda$-ring structure of Proposition \ref{lambdamot} is defined by
\begin{align}
\label{powerdef}
\left(\sum_{n\geq 0} [A_n\rightarrow \AA^1_{\SS}]\cdot T^n\right)^{[B\xrightarrow{g} \AA_{\SS}^1]}:=\sum_{\pi}\left[\Big(\prod_i (B^{\pi_i}\times A_i^{\pi_i})/S_{\pi_i}\Big)\setminus \Delta \xrightarrow{+\,\circ\,\prod_i g^{\pi_i}\times f_i^{\pi_i}}\AA_{\SS}^1\right]\cdot T^{\sum_i i\pi_i},
\end{align}
where the sum is over all functions $\pi\colon \mathbb{N} \to \mathbb{N}$ with finite support and $\Delta$ is the preimage of the \textit{big} diagonal in $\prod_i B^{\pi_i}/S_{\pi_i}$ with respect to the obvious projection. This power structure descends to $\KK^{\mu_n}(\Var/\SS)$ and is preserved by the embeddings $\KK^{\GG_m,n}(\Var/\AA^1_{\SS})\rightarrow \KK^{\GG_m,an}(\Var/\AA^1_{\SS})$ and induces a power structure on $\KK^{\muu}(\Var/\SS)$.
\end{prop}
Given $\pi$ one should think of $\left(\prod_i (B^{\pi_i}\times A_i^{\pi_i})/S_{\pi_i}\right)\setminus \Delta$ as being the configuration space of pairs $(K,\phi)$, where $K$ is a finite subset of $B$ of cardinality $\sum_{i}\pi_i$ and $\phi\colon K \longrightarrow \coprod_i A_i$ is a map sending $\pi_i$ points to $A_i$.
\begin{proof}
The given power structure on $\KK^{\GG_m,n}(\Var/\AA_\SS^1)$ can be checked to be a power structure inducing the given pre $\lambda$-ring structure as in \cite{GLM04}.  The statements regarding the preservation of power structures under embeddings and their descent to the quotient are true due to the correspondence between power structures and pre $\lambda$-rings, and the truth of the corresponding statements on the side of pre $\lambda$-rings --- this is just Proposition \ref{lambdamot} again.
\end{proof}
\begin{rem} This power structure extends to $\KK^{\GG_m,n}(\Sta/\AA^1_\SS)$ inducing a power structure on $\KK^{\muu}(\Sta/\SS)$ which corresponds to the pre $\lambda$-ring structure discussed in Remark \ref{stinclusion}. In order to do this, we have to replace $[B\xrightarrow g \AA^1_\SS]$ and $[A_i\xrightarrow{f_i} \AA^1_\SS]$ with  formal power series $B=\sum_j [B_j\xrightarrow{g_j}\AA^1_\SS] u^j$ and $A_{i}=\sum_k [A_{ik}\xrightarrow{f_{ik}}\AA^1_\SS]u^k$  in $u$ with coefficients in $\KK^{\GG_m,n}(\Var/\AA^1_{\SS})$. To get the correct formula, we should think of these series as being the motives of $\coprod_j B_j \longrightarrow \AA^1_\SS$ resp.\ $\coprod_k A_{ik} \longrightarrow \AA^1_\SS$. The configuration spaces decompose accordingly. If $\pi_{ijk}$ denotes the number of points in $B_j$ mapped into $A_{ik}$, then the correct form of the right hand side of the formula in Equation \ref{powerdef} is
\[ \sum_{\pi}\left[\Big(\prod_{i,j,k} (B^{\pi_{ijk}}_j\times A_{ik}^{\pi_{ijk}})/S_{\pi_{ijk}}\Big)\setminus \Delta \xrightarrow{+\,\circ\,\prod_{i,j,k} g_j^{\pi_{ijk}}\times f_{ik}^{\pi_{ijk}}}\AA_{\SS}^1\right]\cdot u^{\sum_{i,j,k}(j+k)\pi_{ijk}}T^{\sum_{i,j,k} i\pi_{ijk}},\] 
where the sum is now taken over all functions $\pi\colon \mathbb{N}^3\to \mathbb{N}$ with compact support and $\Delta$ denotes the preimage of the big diagonal in $\prod_{i,j,k} B_j^{\pi_{ijk}}/S_{\pi_{ijk}}$ with respect to the obvious projection.  
\end{rem}
\subsection{Motivic Hall algebras}
\label{Hallalg}
We recall the definition of the motivic Hall algebra for the stack of finite-dimensional $A_{Q,W}$-modules, for $A_{Q,W}$ a Jacobi algebra as in Definition \ref{jacdef}.  For $Q$ a finite quiver and $\nn\in\NN^{\verts(Q)}$ a dimension vector, we define the moduli stack
\[
\YY_{Q,\nn}:=\prod_{a\in \edge(Q)}\Hom\left(\Cp^{\nn(t(a))},\Cp^{\nn(s(a))}\right)/\prod_{i\in \verts(Q)}\Gl_{\Cp}(\nn(i)),
\]
where $s(a)$ is the source of the arrow $a$ and $t(a)$ is the target\footnote{In algebraic contexts (as in Section \ref{motdt}) it is generally better to work with right modules, which is why our homomorphisms go from the vector space labelled by the target of the arrow to the vector space labelled by the source.}, and $\Gl_{\Cp}(\nn(i))$ acts by change of basis of $\Cp^{\nn(i)}$.  We define 
\[
\YY_Q:=\coprod_{\nn\in\NN^{\verts(Q)}}\YY_{Q,\nn}.  
\]
If $W\in \mathbb{C} Q/[\mathbb{C} Q,\mathbb{C} Q]$ is a superpotential we define $\XX_{Q,W,\nn}$ to be the Zariski closed subscheme of $\YY_{Q,\nn}$ cut out by the matrix valued equations given by the noncommutative partial differentials (as defined by Equation (\ref{ncder}) and the line following it) of $W$.  We define
\[
\XX_{Q,W}:=\coprod_{\nn\in\NN^{\verts(Q)}}\XX_{Q,W,\nn},
\]
the moduli stack of finite-dimensional modules for $A_{Q,W}$, the Jacobi algebra for $(Q,W)$.  Denote by \begin{align}
\label{nilpd}\XX_{Q,W}^{\nilp}\subset &\XX_{Q,W}\\
\YY_Q^{\nilp}\subset &\YY_Q\nonumber
\end{align}
the stacks\footnote{Note that these stacks do not represent the functor sending a ring $A$ to the groupoid of nilpotent $A_{Q,W}\otimes A$ or $\mathbb{C} Q\otimes A$-modules, flat over $A$.} of finite-dimensional nilpotent right modules for $A_{Q,W}$ and $\mathbb{C} Q$ respectively cut out by the equations $\tr(\rho(c))=0$ for all cyclic paths $c$.
\medbreak
The Abelian groups $\KK(\Sta/\YY_{Q})$, $\KK(\Sta/\YY^{\nilp}_Q)$, $\KK(\Sta/\XX_{Q,W})$ and $\KK(\Sta/\XX^{\nilp}_{Q,W})$ carry Hall algebra products for which the comprehensive reference is the series of papers by Dominic Joyce (see \cite{Jo06a,Jo06b,Jo07a,Jo07b} or also Bridgeland's summary \cite{TB10}).  For completeness we recall the definition.  
\medbreak
We fix our attention on $\KK(\Sta/\YY_Q)$ for now.  Let $[X_i\xrightarrow{f_i}\YY_Q]$ be two effective classes, for $i=0,1$.  The ring $\KK(\Sta/\YY_Q)$ is isomorphic to the inverse limit of the quotients 
\[
\mathfrak{Q}_t:=\KK\left(\Sta/\YY_Q\right)/\KK\left(\Sta/\coprod_{\nn\in\NN^{\verts(Q)}\mbox{ \scriptsize such that }|\nn|\geq t}\YY_{Q,\nn}\right),
\]
by convention (\ref{con2}).  Note that each stack 
\[
\coprod_{\nn\in\NN^{\verts(Q)}\mbox{ \scriptsize such that }|\nn|< t}\YY_{Q,\nn}
\]
is of finite type.  Since the product is linear, we may assume that each morphism $f_i$ factors through an inclusion $\YY_{Q,\nn_i}\hookrightarrow \YY_Q$.  For $a,b\in\NN$, denote by $\Gl_{\Cp}(a,b)$ the Borel subgroup of $\Gl_{\Cp}(a+b)$ preserving the standard flag $0=\Cp^0\subset \Cp^{a}\subset \Cp^{a+b}$.  Let 
\[
\VY_{Q,\nn_0,\nn_1}\subset \VY_{Q,\nn_0+\nn_1}=\prod_{a\in \edge(Q)}\Hom\left(\Cp^{\nn_{0}(t(a))+\nn_1(t(a))},\Cp^{\nn_0(s(a))+\nn_1(s(a))}\right)
\]
be the subspace of points corresponding to linear maps preserving the standard flag 
\[
0=\bigoplus_{i\in \verts(Q)}\Cp^0\subset \bigoplus_{i\in \verts(Q)}\Cp^{\nn_0(i)}\subset \bigoplus_{i\in \verts(Q)}\Cp^{\nn_0(i)+\nn_1(i)},
\]
and let 
\[
\YY_{Q,\nn_0,\nn_1}=\VY_{Q,\nn_0,\nn_1}/\prod_{i\in \verts(Q)}\Gl_{\Cp}(\nn_0(i),\nn_1(i))
\]
be the stack-theoretic quotient.  Then there are three natural morphisms of stacks
\begin{align*}
\pi_1\colon \YY_{Q,\nn_0,\nn_1}\rightarrow &\YY_{Q,\nn_0} \\
\pi_2\colon \YY_{Q,\nn_0,\nn_1}\rightarrow &\YY_{Q,\nn_0+\nn_1}\\
\pi_3\colon \YY_{Q,\nn_0,\nn_1}\rightarrow &\YY_{Q,\nn_1},
\end{align*}
and we define $[X_0\xrightarrow{f_0}\YY_Q]\star[X_1\xrightarrow{f_1}\YY_Q]$ to be the composition given by the top row of the following commutative diagram
\[
\xymatrix{
X_2\ar[r]\ar[d]&\YY_{Q,\nn_0,\nn_1}\ar[d]^-{\pi_1\times \pi_3}\ar[r]^-{\pi_2}\ar[r] &\YY_{Q,\nn_0+\nn_1}\ar@{^{(}->}[r]&\YY_Q\\
X_0\times X_1\ar[r]^-{f_0\times f_1}&\YY_{Q,\nn_0}\times\YY_{Q,\nn_1},
}
\]
where the leftmost square is Cartesian.  This gives consistent well defined products on the quotients $\mathfrak{Q}_t$, and so it gives a well defined product on $\KK(\Sta/\YY_Q)$.  It is easy to see that under the Hall algebra product the group $\KK(\Sta/\YY_Q^{\nilp})$ is a subalgebra.  Similarly, we define $[X_0\xrightarrow{f_0}\XX_{Q,W}]\star_{\KS}[X_1\xrightarrow{f_1}\XX_{Q,W}]$ via the diagram
\[
\xymatrix{
X_2\ar[r]\ar[d]&\XX_{Q,W,\nn_0,\nn_1}\ar[d]^-{\pi_1\times \pi_3}\ar[r]^-{\pi_2}\ar[r] &\XX_{Q,W,\nn_0+\nn_1}\ar@{^{(}->}[r]&\XX_{Q,W}\\
X_0\times X_1\ar[r]^-{f_0\times f_1}&\XX_{Q,W,\nn_0}\times\XX_{Q,W,\nn_1}.
}
\]
\begin{rem}\label{Halldiff}
Note that the group homomorphism $[X\rightarrow \XX_{Q,W}]\mapsto [X\rightarrow \YY_Q]$ induced by the inclusion $\XX_{Q,W}\subset \YY_Q$ is not an algebra homomorphism for these products --- an extension of modules for the Jacobi algebra $A_{Q,W}$, considered as $\mathbb{C} Q$-modules, might not satisfy the relations required to be a $A_{Q,W}$-module.  It is for this reason that we use different notation to distinguish the products $\star_{\KS}$ and $\star$.
\end{rem}
\subsection{Motivic vanishing cycles}
\label{vancycles}
We present some of the ideas expanded upon in greater depth in \cite{Looi00}.  Let $X$ be a smooth scheme over $\Cp$ and let $f\colon X \rightarrow \AA^1_{\Cp}$ be a regular map.  One defines $\Lo_n(X)$, the space of arcs in $X$ of length $n$, to be the scheme representing the functor $Y\mapsto \Hom_{\Sch}(Y\times \Spec(\Cp[t]/t^{n+1}),X)$.  Via the natural inclusion $\Spec(\Cp[t]/t)\rightarrow \Spec(\Cp[t]/t^{n+1})$ there is a map of schemes 
\begin{align*}
p_n\colon \Lo_n(X)\rightarrow X.
\end{align*}
We write $\Lo_n(X)|_{X_0}=p_n^{-1}f^{-1}(0)$.  There is a natural morphism $f_*\colon \Lo_n(X)\rightarrow \Lo_n(\AA_{\Cp}^1)$ given by composition.  An arc in $\AA_{\Cp}^1$ is given by a polynomial $a_0+\ldots+a_nt^n$, and so $\Lo_n(\AA_{\Cp}^1)\cong\AA_{\Cp}^{n+1}$ and the composition of $f_*$ with the projection 
\begin{align*}
\pi\colon &\Lo_n(\AA_{\Cp}^1)\rightarrow\AA^1_{\Cp}
\\&a_0+\ldots+a_n t^n\mapsto a_n
\end{align*}
makes $\Lo_n(X)$ into a scheme over $\AA_{\Cp}^1$.  Moreover there is a $\GG_m$-action on $\Lo_n(X)$ given by rescaling the coordinate $t$ of $\Cp[t]/t^{n+1}$, and $[\Lo_n(X)|_{X_0}\xrightarrow{(\pi\circ f_*) \times p_n} \AA_{X_0}^1]\in \KK^{\GG_m,n}(\Var/\AA_{X_0}^1)$.  We consider the expression
\begin{equation*}
Z_f^{\mathrm{eq}}(T):=\sum_{n\geq 1}\LL^{-(n+1)\dim(X)/2}\cdot [\Lo_n(X)|_{X_0}\xrightarrow{(\pi\circ f_*) \times p_n} \AA_{X_0}^1]T^n
\end{equation*}
as a formal power series with coefficients in $\KK^{\muu}(\Var/ X_0)[\LL^{-1/2}]$.  In general (see \cite[Thm.2.2.1]{DL98}) it makes sense to evaluate this function at infinity, and one defines 
\[
\phi_f=-Z_f^{\mathrm{eq}}(\infty)\in \KK^{\muu}(\Var/X_0)[\LL^{-1/2}], 
\]
the motivic vanishing cycle of $f$.  This definition differs by a factor of $(-\LL^{1/2})^{\dim(X)}$ from the original definition of Denef and Loeser.  This normalisation makes the motivic weights appearing in Donaldson--Thomas theory simpler; the principle is that in Donaldson--Thomas theory and elsewhere, it is best to work with the perverse sheaf of vanishing cycles, which is obtained from the complex of sheaves $\varphi_f\mathbb{Q}_X$ by shifting by half the dimension of $X$.
\medbreak
The motivic vanishing cycle has the property that if $g\colon X_1\rightarrow X_2$ is a smooth morphism of smooth schemes, and $f\colon X_2\rightarrow \AA^{1}_{\Cp}$ is a regular function, then $\phi_{f\circ g}=\LL^{-\dim(g)/2}\cdot g^*\phi_f$.  Given an Artin stack $\mathcal{Z}$ that is a quotient stack $[Z/\Gl_{\Cp}(m)]$ for smooth connected $Z$, and $f\colon \mathcal{Z}\rightarrow \AA_{\Cp}^1$ a function, one defines\footnote{Note that, by relation (\ref{rel1}), $[\BGL_{\Cp}(m)]=[\Gl_{\Cp}(m)]^{-1}\in\KK(\Sta/\Spec(\Cp))$} 
\[
\phi_f=\LL^{m^2/2}\cdot[\BGL_{\Cp}(m)]\cdot \pi_*\phi_{f\circ\pi}\in\KK^{\muu}(\Sta/\mathcal{Z}),
\]
where $\pi\colon Z\rightarrow \mathcal{Z}$ is the projection.
\medbreak
In studying 3-dimensional Calabi--Yau categories, one is often faced with the following situation, which necessitates the use of a \textit{relative} version of motivic vanishing cycles.  Firstly, let $X$ be a finite type scheme, carrying a constructible vector bundle $V$, with a function $f\colon \Tot(V)\rightarrow \Cp$ vanishing on the zero fibre.  By constructible vector bundle, we mean that there is a finite decomposition of $X$ into locally closed subschemes $X=\coprod X_i$, and a vector bundle $V_i$ on each of the $X_i$ (we do not assume that these vector bundles are of the same rank).  By a function on such an object we mean a function on each of the $V_i$, possibly after further decomposition.  In full generality, one should consider formal functions on $V$, by which we mean a function on the formal neighbourhood of the zero section of each of the $V_i$.  We would like to define a motivic vanishing cycle for such a function.  This we do by defining $\Lo_n(V)$ to be the space of those arcs in $\Tot(V)$ that restrict to a single fibre of the projection $\pi\colon V\rightarrow X$.  More precisely, we define $\Lo_n(V)$ via the Cartesian diagram
\[
\xymatrix{
\Lo_n(V)\ar[r]\ar[d]&\Lo_n(\Tot(V))\ar[d]^{\tau_*}\\
X\ar[r]^{\beta}&\Lo_n(X)
}
\]
where $\tau_*$ is induced by the projection $\tau\colon\Tot(V)\rightarrow X$, and the map $\beta$ is the inclusion of constant arcs.  We define $\Lo_n(V)|_{X}=p_n^{-1}(X)$ as before.  Finally, define 
\begin{equation}
\label{zetaf}
Z_f^{\mathrm{eq}}(T):=\sum_{n\geq 1}\LL^{-(n+1)\rank(V)/2}\cdot [\Lo_n(V)|_{X}\xrightarrow{ (\pi\circ f_*)\times p_n} \AA_{X}^1]T^n\in \KK^{\muu}(\Var/ X)[\LL^{-1/2}].
\end{equation}
We claim that the definition 
\[
\phi_f^{\rel}:=Z_f^{\mathrm{eq}}(\infty)
\]
makes sense, in other words, that the relative zeta function (\ref{zetaf}) can be evaluated at infinity.  The claim is justified using Kontsevich's transformation formula (see \cite[Sec.3]{Looi00}) in the same way as \cite[Thm.2.2.1]{DL98}.  In a little more detail, by Hironaka's theorem, we may find an embedded resolution $Y\xrightarrow{g} \Tot(V)$ of $f^{-1}(0)$, considered as a subvariety of $\Tot(V)$, blowing up along smooth centers $H_1,\ldots ,H_n$.  I.e. we have that $(fg)^{-1}(0)$ is a normal crossings divisor.  After replacing $X$ by a Zariski open subvariety $X'\subset X$, we may assume that each projection from $H_i$ to $X$ is smooth.  Define $Y'=g^{-1}(X')$, then possibly after shrinking $X'$ further, we may assume that $(fg)^{-1}(0)$ is a smooth family of normal crossing divisors.  Now the claim (over $X'$) follows by the proof of \cite[Thm.5.4]{Looi00}, and the discussion following it.  Finally, we consider the complement $X\setminus X'$, which can be decomposed into finitely many smooth schemes $X'=\coprod X_i$ of dimension strictly less than $\dim(X)$ --- the general result follows by Noetherian induction and the cut and paste relations.


\section{Motivic Donaldson--Thomas theory}
\label{motdt}
\subsection{Three-dimensional Calabi--Yau categories}
We recall the essential ingredients of the theory of motivic Donaldson--Thomas invariants from \cite{KS}.  We will begin with the data that one feeds into this machine.  One starts with $\CC$, a 3-Calabi--Yau category. By a 3-Calabi--Yau category $\CC$ we mean a set of objects $\ob(\CC)$, between any two objects $x_i,x_j\in\ob(\CC)$ a $\ZZ$-graded vector space $\Hom_{\CC}(x_i,x_j)$, and a countable collection of operations
\[
b_{\CC,n}\colon \Hom_{\CC}(x_{n-1},x_n)[1]\otimes\ldots \otimes\Hom_{\CC}(x_0,x_1)[1]\rightarrow \Hom_{\CC}(x_0,x_n)[1]
\]
of degree 1, satisfying the condition
\[
\sum_{\alpha+\beta+\gamma=n}b_{\CC,\alpha+1+\gamma}\circ(\one^{\otimes \alpha}\otimes b_{\CC,\beta}\otimes\one^{\otimes \gamma})=0.
\]
See \cite{KLH} for a comprehensive guide to $A_{\infty}$-categories, or \cite{kaj03} for a similarly comprehensive guide to cyclic $A_{\infty}$-categories, or \cite{keller-intro} for a gentle and concise reference for most of what follows.  All these ideas are also covered in the notes \cite{KSnotes}.  The 3-Calabi--Yau condition consists of the extra data of a skewsymmetric nondegenerate bracket 
\[
\langle\bullet,\bullet\rangle_{\CC}\colon \Hom_{\CC}(x_i,x_j)[1]\otimes\Hom_{\CC}(x_j,x_i)[1]\rightarrow \Cp
\]
of degree -1, such that the functions $W_{\CC,n}:=\langle b_{\CC,n-1}(\bullet,\ldots,\bullet),\bullet\rangle$ are cyclically symmetric.  One defines 
\[
W_{\CC}(z):=\sum_{n\geq 2}W_{\CC,n}(z,\ldots,z)/n, 
\]
a formal function on $\Hom_{\CC}^1(x_i,x_i)$ for each $x_i\in\ob(\CC)$.
\medbreak
In this section we will recall the definition of a particular 3-Calabi--Yau category $\tw(\DD(Q_{-2},W_d))$, the $A_{\infty}$-\textit{category of twisted complexes} over a certain 3-Calabi--Yau category $\DD(Q_{-2},W_d)$ built out of the same data $(Q_{-2},W_d)$ as $A_{Q_{-2},W_d}$.  The category $\tw(\DD(Q_{-2},W_d))$ will be shown to be a 3-Calabi--Yau enrichment of the category $\Derbe(Y_d)$, the derived category of coherent sheaves on $Y_d$ with bounded total cohomology, set-theoretically supported on the exceptional locus $C_d\subset Y_d$, in the sense that there is a composition of equivalences of categories, beginning with the homotopy category of $\tw(\DD(Q_{-2},W_d))$:
\begin{equation}
\label{bp}
\xymatrix@C+1.3pc{
  &\Derf(\rmod\hat{A}_{Q_{-2},W_d}) \ar[r]^{\simeq} &\Dern(\rmod A_{Q_{-2},W_d}) \ar[r]_-{\simeq}^-{\substack{\mathrm{VdB's}\\\mathrm{equivalence}}}&\Derbe(Y_d)\\
  \Ho(\tw(\DD(Q_{-2},W_d)))\ar[r]_{\simeq}^-{\substack{\mathrm{Koszul}\\\mathrm{duality}}}&\Derf(\rmod\Gamma(Q_{-2},W_d))\ar[u]^{\simeq} \ar[r]^{\simeq} &\Dern(\rmod\Gamma(Q_{-2},W_d))\ar[u]^{\simeq}.
}
\end{equation}
\medbreak
The algebra $\hat{A}_{Q_{-2},W_d}$ is the Jacobi algebra defined as in Definition \ref{jacdef}, completed at the ideal generated by the arrows of $Q_{-2}$.  The category $\Derf(\rmod\hat{A}_{Q_{-2},W_d})$ is the derived category of right $\hat{A}_{Q_{-2},W_d}$-modules with finite-dimensional total cohomology, and $\Dern(\rmod A_{Q_{-2},W_d})$ is the derived category of $A_{Q_{-2},W_d}$-modules with nilpotent finite-dimensional total cohomology.  As diagram (\ref{bp}) indicates, the story starts with Koszul duality, so we start our exposition with the Koszul dual of $\DD(Q_{-2},W_d)$, which is the \textit{Ginzburg differential graded category}.
\medbreak
Given the data of a quiver with potential $(Q,W)$ in \cite{ginz} Ginzburg defines the dg category $\Gamma(Q,W)$. It is constructed as follows.  The quiver $Q$ defines a bimodule $S$ for the semisimple ring $R:=\Cp^{\verts(Q)}$, where we set
\[
\dim(e_i\cdot S\cdot e_j):=\#\mbox{ of arrows from }j\mbox{ to }i.
\]
The objects of the category $\Gamma(Q,W)$ are just the vertices of the quiver, i.e.\ $\ob(\Gamma(Q,W)):=\verts(Q)$, and for two vertices $x_i,x_j$ we put 
\[ \Hom_{\Gamma(Q,W)}(x_i,x_j)= e_j\cdot T_R\left(R[2]\oplus S^{\vee}[1] \oplus S\right) \cdot e_i, \]
where $e_i,e_j\in R$ are the idempotent elements corresponding to $x_i$ resp.\ $x_j$, and $S^{\vee}$ is the dual of $S$ in the category of $R$-bimodules. Moreover, $T_R(M)$ denotes the completion of the free unital algebra object generated by $M$ in the category of $R$-bimodules. Composition in the category $\Gamma(Q,W)$ is given by the tensor product in the category of $R$-bimodules. We define a differential $d$ of degree one on $T_R(R[2] \oplus S^{\vee}[1] \oplus S)$ satisfying the Leibniz rule and such that 
\begin{align*}
d(e_i[2])=&\sum_{a_k\colon x_i \to x_j} a_k^*a_k-\sum_{a_k\colon x_j \to x_i}a_k a_k^*\\
d(a_k^*[1])=&\partial W/\partial a_k\\
d(a_k)=&0, 
\end{align*}
where $a_k$ (resp.\ $a_k^*$) runs through a basis (resp. dual basis) of the vector space $e_jSe_i$ (resp.\ $e_iS^{\vee}e_j=(e_jSe_i)^{\vee}$). This makes $\Gamma(Q,W)$ into a dg-category and hence into an $A_{\infty}$-category.  For certain choices of $(Q,W)$, including our choice $(Q_{-2},W_d)$, the Ginzburg differential graded category $\Gamma(Q,W)$ has cohomology concentrated in degree zero.  Moreover, for any choice of $(Q,W)$, there is an isomorphism
\[
\Ho^0(\Gamma(Q,W))\cong \hat{A}_{Q,W}
\]
where $\hat{A}_{Q,W}$ is the Jacobi algebra defined as in Definition \ref{jacdef}, completed at the ideal generated by the arrows of $Q$, and considered in the usual way as a category whose objects are the idempotents $e_i$.  There is a natural equivalence of categories between finite-dimensional modules over $\hat{A}_{Q,W}$ and nilpotent finite-dimensional modules over $A_{Q,W}$.  Together, these facts provide the central commutative square of equivalences in (\ref{bp}).
\medbreak
As for $\Gamma(Q,W)$, the objects of the category $\DD(Q,W)$ are defined to be the vertices of the quiver, i.e.
\[
\ob(\DD(Q,W)):=\verts(Q).  
\]
The homomorphism spaces between these objects are graded vector spaces concentrated in degrees between zero and three.  One sets
\begin{equation}
\Hom^n_{\DD(Q,W)}(x_i,x_j):=\begin{cases} \Cp^{\delta_{ij}} &\mbox{ if }n=0\\ (e_i\cdot S\cdot e_j)^{\vee}&\mbox{ if }n=1\\(e_j\cdot S\cdot e_i) &\mbox{ if }n=2\\ (\Cp^{\vee})^{\delta_{ij}}&\mbox{ if } n=3, \end{cases}
\end{equation}
where $\delta_{ij}$ is the Kronecker delta function and $\Cp^{\vee}\cong\Cp$ is the vector dual of the one dimensional complex vector space $\Cp$.  The $A_{\infty}$ operations on this category are given by firstly setting the natural generator $1_{i}$ of $\Hom_{\DD(Q,W)}^0(x_i,x_i)$ to be a strict unit for every $i\in Q_0$.  This means that $b_2(f,1_i)=f$ and\footnote{The strange sign here is the price we pay for considering the maps $b_n\colon \Hom_{\CC}(x_{n-1},x_n)[1]\otimes\ldots \otimes\Hom_{\CC}(x_0,x_1)[1]\rightarrow \Hom_{\CC}(x_0,x_n)[1]$ instead of $m_n\colon \Hom_{\CC}(x_{n-1},x_n)\otimes\ldots \otimes\Hom_{\CC}(x_0,x_1)\rightarrow \Hom_{\CC}(x_0,x_n)$.  The payoff is that there are a lot fewer signs overall.} $b_2(1_i,g)=-g$ for all $f\in\Hom_{\DD(Q,W)}(x_i,x_j)$ and $g\in\Hom_{\DD(Q,W)}(x_j,x_i)$, and any insertion of $1_i$ into $b_n$ for any $n\geq 3$ results in the zero function.  We let $b_{\DD(Q,W),2}(\theta, z)=-\theta(z)1^*_j$ with $1^\ast_j\in \Hom^3_{\CC}(x_j,x_j)$ being the dual basis of $1_j$, and $b_{\DD(Q,W),2}(z,\theta)=\theta(z)1^*_i$ for any $\theta\in \Hom^1_{\DD(Q,W)}(x_i,x_j)$ and $z\in\Hom^2_{\DD(Q,W)}(x_j,x_i)$.  Then for degree reasons all that is left is to define the degree one operations 
\[
b_{\DD(Q,W),m}\colon\Hom^1_{\DD(Q,W)}(x_{m-1},x_m)[1]\otimes\ldots\otimes\Hom^1_{\DD(Q,W)}(x_0,x_1)[1]\rightarrow \Hom^2_{\DD(Q,W)}(x_0,x_m)[1]
\]
which are given by $W_{m+1}$, the $(m+1)$th homogeneous part of $W$, via the natural pairing
\[
(e_{m-1}\cdot S\cdot e_{m})^{\vee}\otimes\ldots\otimes (e_0\cdot S\cdot e_1)^{\vee}\otimes e_0\cdot S\cdot e_1\otimes\ldots\otimes e_{m-1}\cdot S\cdot e_m\otimes e_{m}\cdot S\cdot e_0\rightarrow e_m\cdot S\cdot e_0.
\]
Note that this definition results in the identity $W=W_{\DD(Q,W)}|_{\End^1(\bigoplus_{i\in \verts(Q)}x_i)}$.  
\medbreak
The category $\DD(Q,W)$ has a natural inner product $\Hom_{\DD(Q,W)}(x_i,x_j)[1]\otimes \Hom_{\DD(Q,W)}(x_j,x_i)[1]\rightarrow \Cp[-1]$ satisfying the cyclicity condition. 
\medbreak
We now come to the connection between $\Gamma(Q,W)$ and $\DD(Q,W)$. This is explained via Koszul duality for $A_{\infty}$-algebras, for which an excellent reference is \cite{koszul}.  Using the projection to the degree zero part of $\DD(Q,W)$, we can make $R_i:=e_iR\cong\mathbb{C}$ into a (trivial) right $\DD(Q,W)$-module, which we will denote $R_{i,\DD(Q,W)}$, and we get an object in $\rmod(\DD(Q,W)):= \Fun_\infty(\DD(Q,W)^{\op},\Vect^\ZZ_\mathbb{C})$, the dg category of right $A_{\infty}$-modules over $\DD(Q,W)$ with finite-dimensional bounded cohomology. With the help of the bar construction one can show that there are quasi-isomorphisms 
\[
\Hom_{\rmod\DD(Q,W)}(R_{i,\DD(Q,W)},R_{j,\DD(Q,W)})\simeq\Hom_{\Gamma(Q,W)}(x_j,x_i)\simeq\Hom_{\rmod\Gamma(Q,W)}(x_i,x_j),
\]
where we used the Yoneda embedding of $\Gamma(Q,W)$ into $\rmod\Gamma(Q,W)$ for the final quasi-isomorphism --- in fact if one uses the (reduced) bar construction to demonstrate the first of these quasi-isomorphisms, it is an equality.  This establishes that $\Gamma(Q,W)$ and $\DD(Q,W)$ are Koszul dual, and so by \cite[Thm.2.4]{koszul} we get similar quasi-isomorphisms after swapping them --- i.e. there is a quasi-isomorphism 
\[
\Hom_{\rmod\Gamma(Q,W)}(R_{i,\Gamma(Q,W)},R_{j,\Gamma(Q,W)})\simeq\Hom_{\DD(Q,W)}(x_i,x_j).  
\]
Hence the induced functor between homotopy categories $\RHom(\rmod R_{\Gamma(Q,W)},-)\colon\DER(\rmod\Gamma(Q,W))\rightarrow \DER_{\infty}(\rmod\DD(Q,W))$ takes $R_{i, \Gamma(Q,W)}$ to a module quasi-isomorphic to $x_i$, considered as a $\DD(Q,W)$ module, and restricting, we obtain the diagram of functors
\begin{equation}
\label{kdiag}
 \xymatrix @C=3cm { \DER(\rmod\Gamma(Q,W)) \ar[r]^{\RHom(R_{\Gamma(Q,W)},-)} & \DER_{\infty}(\rmod\DD(Q,W)) \\ 
\DER(\langle R_{i,\Gamma(Q,W)},i\in \verts(Q)\rangle_{\thick}) \ar@{^{(}->}[u] \ar[r]^{\simeq} & \DER_{\infty}(\langle x_{i}\in\rmod\DD(Q,W),i\in\verts(Q)\rangle_{\thick}) \ar@{^{(}->}[u] \\
\DER(\langle R_{i,\Gamma(Q,W)},i\in \verts(Q)\rangle_{\triang}) \ar@{^{(}->}[u]^{\simeq}_{\alpha} \ar[r]^{\simeq} & \DER_{\infty}(\langle x_{i}\in\rmod\DD(Q,W),i\in\verts(Q)\rangle_{\triang}) \ar@{^{(}->}[u]^{\simeq}_{\beta} }
\end{equation}
where for $S\subset \ob(\rmod\Gamma(Q,W))$, the category $\DER(\langle S\rangle_{\triang})$ is the full subcategory of the derived category of $\Gamma(Q,W)$-modules $M$ that are quasi-isomorphic to objects in the closure of $S$ under taking triangles and shifts, and $\DER(\langle S\rangle_{\thick})$ is defined in the same way, except we take the closure under the operation of taking retracts too.  
\medbreak
The lowest two horizontal functors in (\ref{kdiag}) are equivalences by Koszul duality for module categories \cite[Thm.5.4]{koszul}.  The inclusion $\alpha$ is a equivalence, since its source and target can both be seen to be the full subcategory of the derived category of $\Gamma(Q,W)$-modules consisting of dg modules with finite-dimensional nilpotent total cohomology.  In particular, $\DER(\langle R_{i,\Gamma(Q,W)},i\in \verts(Q)\rangle_{\triang})$ is already closed under taking retracts.  It follows that $\beta$ is an equivalence too.

\medbreak
The point of introducing the category of twisted complexes $\tw(\DD(Q,W))$ is that it is a category that is a natural 3-Calabi-Yau enrichment of $\DER_{\infty}(\langle x_{i}\in\rmod\DD(Q,W),i\in\verts(Q)\rangle_{\triang})$, which by (\ref{kdiag}) and (\ref{ncDE}) is equivalent to $\Derbe(Y_d)$ in the case $(Q,W)=(Q_{-2},W_d)$.  We refer to \cite[Sec.7]{keller-intro} for a comprehensive account of the category of twisted complexes, and here recall its main features.
\medbreak
Objects of $\tw(\DD(Q,W))$ are given by pairs $(T,\alpha)$, where $T=\bigoplus_{i=1}^n x_{a_i}[b_i] \in \rmod\DD(Q,W)$ is a finite direct sum of right $\DD(Q,W)$-modules given by integer shifts of objects $x_{a_i}\in \ob(\DD(Q,W))$ covariantly embedded via the Yoneda embedding, and $\alpha$ is an element of 
\[
\bigoplus_{i<j} \Hom_{\DD(Q,W)}^{b_j-b_i+1}(x_{a_i},x_{a_j})\simeq\bigoplus_{i<j} \Hom_{\DD(Q,W)}^1(x_{a_i}[b_i],x_{a_j}[b_j]) \subset \Hom_{\DD(Q,W)}(T,T) 
\]
satisfying the Maurer--Cartan equation
\[
\sum_{n\geq 1} b_{\DD(Q,W),n}(\alpha,\ldots,\alpha)=0.
\]
Given two pairs $(T_1,\alpha_1)$ and $(T_2,\alpha_2)$, where $T_1=\bigoplus_{i\in I}x_{a_{1,i}}[b_{1,i}]$ and $T_2=\bigoplus_{j\in J}x_{a_{2,j}}[b_{2,j}]$, we define the graded vector space 
\[
\Hom_{\tw(\DD(Q,W))}((T_1,\alpha_1),(T_2,\alpha_2)):=\bigoplus_{i,j}\Hom_{\DD(Q,W)}(x_{a_{1,i}},x_{a_{2,j}})[b_{2,j}-b_{1,i}]\simeq \Hom_{\DD(Q,W)}(T_1,T_2).  
\]
Multiplication is twisted by setting
\[
b_{\tw(\DD(Q,W))}(f_n,\ldots,f_1)=\sum b_{\DD(Q,W)}(\alpha_n,\ldots,\alpha_n,f_n,\alpha_{n-1},\ldots,\alpha_1,f_1,\alpha_0,\ldots,\alpha_0)
\]
where $f_i \in\Hom_{\tw(\DD(Q,W))}((T_{i-1},\alpha_{i-1}),(T_i,\alpha_i))$.  One may check that this is again a 3-Calabi--Yau category.  For 
\begin{align*}
f\in&\Hom_{\tw(\DD(Q,W))}\left((\bigoplus_{i\in I} x_{a_{1,i}}[b_{1,i}],\alpha_1),(\bigoplus_{j\in J} x_{a_{2,j}}[b_{2,j}],\alpha_2)\right)
\\g\in&\Hom_{\tw(\DD(Q,W))}\left((\bigoplus_{j\in J} x_{a_{2,j}}[b_{2,j}],\alpha_2),(\bigoplus_{i\in I} x_{a_{1,i}}[b_{1,i}],\alpha_1)\right)
\end{align*}
one sets \[
\langle f,g\rangle:=\sum_{i\in I,j\in J}\langle f_{ij},g_{ji}\rangle,
\]
where we denote by $f_{ij}$ the degree $(b_{2,j}-b_{1,i})$ morphism $x_{a_{1,i}}\rightarrow x_{a_{2,j}}$ induced by $f$, and define $g_{ji}$ similarly.
\medbreak 
In fact we will only be interested in the category $\tw_0(\DD(Q,W))$, which we define to be the full subcategory of $\tw(\DD(Q,W))$ with objects given by pairs $(T,\alpha)$, with $T$ isomorphic to a finite direct sum of \textit{unshifted} copies of the right modules $x_i\in\ob(\DD(Q,W))$.  Under $\RHom_{\rmod\Gamma(Q,W)}(R_{\Gamma(Q,W)},-)$ this in turn is an enrichment of the Abelian category of finite-dimensional nilpotent modules over the Jacobi algebra $A_{Q,W}$.
\medbreak
Let $\TF_{\nn}$ be the moduli functor on finite type schemes defined as follows
\begin{align*}
\TF_{\nn}(X):=&\{\mbox{pairs of rank $\sum_{i\in \verts(Q)}\nn(i)$ vector bundles }\bigoplus_{i\in \verts(Q)}T_{\nn(i)} \mbox{ on }X \mbox{ and }\\&\alpha\in\bigoplus_{i,j\in \verts(Q)} \Hom^1_{\DD(Q,W)}(x_i,x_j)\otimes T_{\nn(i)}^*\otimes T_{\nn(j)}\\ &\mbox{such that }\sum_{n\geq 1} b_{\DD(Q,W),n}(\alpha,\ldots,\alpha)=0\}.
\end{align*}
This moduli functor takes schemes over $\Cp$ to sets of families of objects in $\tw_0(\DD(Q,W))$.  This is naturally made into a groupoid valued moduli functor, where the morphisms are defined via the conjugation action of $\prod_{i\in \verts(Q)}\Gl_{\Cp}(\nn(i))$.  There is a natural isomorphism of moduli functors $\TF_{\nn}\rightarrow \TN_{\nn}$, where
\begin{align*}
\TN_{\nn}(X):=&\{\mbox{vector bundles }T \mbox{ on }X\mbox{ with a }\OO_X\otimes \Gamma(Q,W)\mbox{-action, nilpotent} \\&\mbox{with respect to the }\Gamma(Q,W)\mbox{ factor, such that for all } i\in \verts(Q),\\& T\cdot e_i \mbox{ is a rank } \nn(i)\mbox{ vector bundle.}\}
\end{align*}
The moduli functor $\TN_{\nn}$ is again a groupoid valued functor with morphisms given by conjugation, and its groupoid of geometric points is the same as for the stack $\XX^{\nilp}_{Q,W,\nn}$.  \medbreak

\subsection{Orientation data}
There is one extra piece of data, aside from the 3-Calabi--Yau category \\$\tw_0(\DD(Q_{-2},W_d))$, that we need before we can apply the machinery of \cite{KS} to define and compute motivic Donaldson--Thomas invariants of (-2)-curves, which is the data $(\LLL,\phi)$ of an ind-constructible super (i.e. $\mathbb{Z}_2$-graded) line bundle $\LLL$ on $\XX^{\nilp}_{Q_{-2},W_d}$ along with a chosen trivialisation of the tensor square 
\[
\phi\colon \LLL^{\otimes 2}\cong \one_{\XX^{\nilp}_{Q_{-2},W_d}}.
\]  
\medbreak
Note that every constructible super line bundle $\LLL$ on a scheme $X$ has trivial tensor square, since up to constructible decomposition of the base $X$ we can write 
\[
\LLL\cong \one_{X_{\textrm{even}}}\oplus\one_{X_{\textrm{odd}}}[1], 
\]
where $X=X_{\textrm{even}}\coprod X_{\textrm{odd}}$ is the constructible decomposition of $X$ defined by the constructible function on $X$ provided by taking the parity of $\LLL$.  So all the data here is in this choice of trivialisation (and the parity of the super line bundle $\LLL$).  Such data is required to satisfy a cocycle condition (see Section 5.2 of \cite{KS}), ensuring that the integration map defined with respect to it (see Equation (\ref{intdef})) is a $\KK(\Sta/\Spec(\Cp))$-algebra homomorphism, and is called orientation data in \cite{KS}.  An isomorphism of orientation data is just an isomorphism of the underlying constructible super line bundles commuting with the trivialisations of the squares.  Isomorphic choices will give rise to the same integration map (\ref{intdef}).  In fact for $Q,W$ a finite quiver with arbitrary potential, $\XX_{Q,W}^{\nilp}$ comes with a natural choice of orientation data, which we briefly describe; more details can be found in \cite[Sec.7.1]{Dav10a}.
\medbreak
Given an object $\eta=(T,\alpha)$ of $\tw_0(\DD(Q,W))$, there is an explicit model of the cyclic $A_{\infty}$-algebra $\End_{\tw(\DD(Q,W))}(\eta)$, coming from the definition of $\tw(\DD(Q,W))$.  In particular, there is a differential\footnote{Recall that $\End^{\bullet}_{\tw(\DD(Q,W)))}(\eta)\cong\End^{\bullet}_{\DD(Q,W)}(T)$ as graded vector spaces, and does not depend on $\alpha$, so in fact we obtain a family of differentials as we vary $\alpha$.} $b_{\alpha,1}$ on $\End^{\bullet}_{\tw(\DD(Q,W)))}(\eta)$, and a nondegenerate inner product on $\End^1_{\tw(\DD(Q,W))}(T)/\Ker(b_{\alpha,1})$ given by $\langle b_{\alpha,1}(\bullet),\bullet\rangle$.  Across the family of possible $\alpha$ in the pair $(T,\alpha)$, given by solutions to the Maurer--Cartan equation, we obtain a constructible vector bundle 
\begin{equation}
\label{quadhome}
\End^1_{\tw(\DD(Q,W))}(T)/\Ker(b_{\tw(\DD(Q,W)),1})
\end{equation}
with nondegenerate quadratic form which we will denote by $\overline{Q}$.  It is only a constructible vector bundle since the dimension of (\ref{quadhome}) jumps, due to the dependancy of $b_{\tw(\DD(Q,W)),1}$ on $\alpha$.  Given a constructible super vector bundle $\VV$ on a stack $\SS$, one defines the superdeterminant 
\[
\sDet(\VV):=\prod  ^{\dim(\VV)}\left(\bigwedge^{\mathrm{top}}\VV_{\mathrm{even}}\otimes  \bigwedge^{\mathrm{top}}\VV_{\mathrm{odd}}^*\right),
\]
where here $\prod$ denotes the change of parity functor.  Say now $\VV$ has nondegenerate quadratic form $\overline{Q}_{\VV}$, then we obtain a trivialisation of $\sDet(\VV)^{\otimes 2}$ since $\overline{Q}_{\VV}$ establishes an isomorphism $\sDet(\VV)\cong\sDet(\VV)^*$.  In the present situation, orientation data on $\XX^{\nilp}_{Q_{-2},W_d}$, considered as the moduli space of objects in $\tw_0(\DD(Q_{-2},W_d))$, is provided by the superdeterminant of (\ref{quadhome}), with the trivialisation of the tensor square provided by the nondegenerate inner product $\langle b_{\tw(\DD(Q_{-2},W_d)),1}(\bullet),\bullet\rangle$.  We will denote this choice of orientation data by $\tau_{Q_{-2},W_d}$.\medbreak
\begin{rem}\cite[Thm.8.3.1]{Dav10a}
\label{ODchoice}
There are in general several choices for the orientation data of a 3-Calabi--Yau category.  However in the case of the category $\tw_0(\DD(Q,W))$ this range of choices is quite small, due to the constraint that the orientation data must satisfy the cocycle condition from \cite{KS}.  In fact the orientation data is determined up to isomorphism entirely by its restriction to the simple modules $x_i$, for $i\in \verts(Q)$, and so one deduces that there are $2^{\verts(Q)}$ isomorphism classes of choices, giving rise to $2^{\verts(Q)}$ distinct integration maps, defined as in Theorem \ref{intmappres}.
\end{rem}
\begin{defn}
A constructible 3-Calabi--Yau vector bundle on a scheme $X$ is a constructible $\mathbb{Z}$-graded vector bundle $V$, along with degree one morphisms $b_{n}(V[1])^{\otimes n}\rightarrow V[1]$ and a degree minus one morphism $\langle \bullet,\bullet\rangle\colon V[1]\otimes V[1]\rightarrow \one_X$ satisfying the same conditions as a 3-Calabi--Yau category.
\end{defn}
We recall the definition of a morphism of cyclic $A_{\infty}$-objects in the case of a constructible 3-Calabi--Yau vector bundle.
\begin{defn}
A morphism $f\colon V\rightarrow V'$ of constructible 3-Calabi--Yau vector bundles is a countable collection of morphisms of constructible vector bundles $f_n\colon V[1]^{\otimes n}\rightarrow V'[1]$ satisfying the conditions
\[
\sum_{\alpha+\beta+\gamma=n}f_{\alpha+1+\gamma}(\one^{\otimes\alpha}\otimes b_{\beta}\otimes\one^{\otimes\gamma})=\sum_{n=\alpha_1+\ldots+\alpha_s}b'_{s}(f_{\alpha_1}\otimes\ldots\otimes f_{\alpha_s}),
\]
for all $n$ as well as the extra conditions that $\langle\bullet,\bullet\rangle_{V'}\circ f_1\otimes f_1=\langle\bullet,\bullet\rangle_{V}$ and $\sum_{a+b=n}\langle \bullet,\bullet\rangle_{V'}\circ f_a\otimes f_b=0$ for all $n\geq 3$.
\end{defn}
To complete the definition of the integration map of \cite{KS} we need the following proposition.
\begin{prop}\cite[Thm.5.15]{kaj03}
\label{cmmt}
There is a locally constructible formal isomorphism of cyclic $A_{\infty}$-vector bundles \[
(\End^{\bullet}_{\tw(\DD(Q_{-2},W_d))},b_{\tw(\DD(Q_{-2},W_d))})\cong (\Ext^{\bullet}_{\tw(\DD(Q_{-2},W_d))}\oplus V^{\bullet},b')
\]
on the stack $\XX^{\nilp}_{Q_{-2},W_d}$ such that $b'_1$ factors via a map $V^{\bullet}\rightarrow V^{\bullet}$, for $i\geq 2$ $b'_i$ factors via a map $(\Ext^{\bullet}_{\tw(\DD(Q_{-2},W_d))})^{\otimes i}\rightarrow \Ext^{\bullet}_{\tw(\DD(Q_{-2},W_d))}$, and $(V^{\bullet},b'_1)$ is an acyclic complex.  This splitting is unique up to isomorphisms of cyclic $A_{\infty}$-vector bundles.
\end{prop}
Note that even though one starts with the data of a cyclic $A_{\infty}$-vector bundle, the splitting will only take place in the category of locally constructible cyclic $A_{\infty}$-vector bundles, since the dimension of the kernel of $b_1$ will jump in families.  The reference \cite{kaj03} demonstrates this splitting in the case of cyclic $A_{\infty}$-categories; which for us corresponds to the case in which the base of the cyclic $A_{\infty}$-vector bundle is a point.  In fact the proof produces a \textit{canonical} decomposition, once a \textit{choice} of contracting homotopy is made.  Since after constructible decomposition we can construct a contracting homotopy for $b_{\tw(\DD(Q_{-2},W_d)),1}$, the version of the proposition stated above is indeed a consequence of \cite[Thm.5.15]{kaj03}.
\medbreak
Given a constructible 3-Calabi--Yau vector bundle $V$, we define the function $W_{\min}$ as follows.  Firstly, let $E$ be a cyclic minimal model for $V$.  Next, consider the Artin stack $E^1/E^0$ over $X$, given over a point $x\in X$ by taking the stack theoretic quotient of the trivial action of $E_x^0$ on $E_x^1$ (this is an example of a cone stack, see e.g. \cite{NormalCone}, where they arise in a similar context).  We define $W_{\min}$ to be the function on this stack defined by $W_{E}$, the potential for the minimal part $E$.  This potential is strictly speaking only defined up to a formal automorphism, which will not matter when it comes to considering motivic vanishing cycles, as a result of the following proposition. 
\begin{prop}
\label{formalaut}
Let $V$ be a vector bundle on a scheme $X$, and let $f$ be a formal function on $V$ with trivial constant coefficient (i.e.\ $f$ vanishes on $X$, considered as the zero section of $V$) such that $\phi^{\rel}_f|_X$ is well defined.  Let $g$ be another formal function on $V$ vanishing on $X$, such that there exists a formal change of coordinates $q$ on the vector bundle $V$ around $X$ such that $f=g\circ q$, considered as functions on a formal neighbourhood of $X$.  Then $\phi^{\rel}_g|_X$ is well defined and $\phi^{\rel}_f|_X=\phi^{\rel}_g|_X$.
\end{prop}
Note that since we are dealing here with functions defined on vector bundles, we use the relative version of motivic vanishing cycles introduced at the end of Section \ref{vancycles}.  
\begin{proof}
The proposition follows straight from the definition, since $q$ induces $\GG_m$-equivariant isomorphisms on arc spaces making the following diagram commute
\[
\xymatrix{
\Lo_n(V)|_{X}\ar[rd]_{\pi\circ f_*}\ar[rr]^{q_*}&&\Lo_n(V)|_{X}\ar[ld]^{\pi\circ g_*}\\
&\AA_{\Cp}^1,
}
\]
where $\pi$ is as in Section \ref{vancycles}, and $\Lo_n(V)$ is as at the end of the same section.
\end{proof}
As in \cite{KS} we would like to identify the relative motivic vanishing cycles $\phi_{Q_1}^{\rel}$, $\phi_{Q_2}^{\rel}$ of quadratic functions $Q_1$, $Q_2$ on constructible vector bundles $V_1$, $V_2$ under the conditions that taking the parity of $V_1$ or $V_2$ gives the same element of $\Gamma_{\constr}(X,\mathbb{Z}_2)$, the group of of constructible $\mathbb{Z}_2$-valued function on $X$, and taking the determinant of $Q_1$ or $Q_2$ gives the same element of $\Gamma_{\constr}(X,\Cp^*)/(\Gamma_{\constr}(X,\Cp^*))^2$.  The reason this is desirable is that it means that we only have to keep track of these two pieces of data for the pair $V,Q$ to know what the relative motivic vanishing cycle $\phi^{\rel}_Q$ is, and the reason this identification is justifiable is that this identification becomes trivial after taking realisations of motives (for example Hodge polynomials, etc.).  This we achieve as follows: we impose the extra relation in $\KK^G(\Var/\SS)$ given by identifying
\[
[X/_{\rho_1}H\xrightarrow{f/\rho_1}\SS]-[X/_{\rho_2}H\xrightarrow{f/\rho_2}\SS]
\]
for all smooth $X$, for all $H$-actions $\rho_1$, $\rho_2$ for $H$ a finite group satisfying the property that the $G$-equivariant function $f$ is $H$-invariant, and that the induced $H$-actions on the cohomology of a fibre over $x$, for any $x\in\SS$, are the same.  One may easily check that the pre $\lambda$-ring structure on $\KK^{\muu}(\Sta/\AA^1_{\SS})$ descends to a pre $\lambda$-ring structure on $\overline{\KK}^{\muu}(\Sta/\AA^1_{\SS})$.

Let $Q,W$ be a quiver with polynomial potential.  Given an element $[X\xrightarrow{f}\XX^{\nilp}_{Q,W,\nn}]\in \KK(\Var/\XX^{\nilp}_{Q,W,\nn})$ Kontsevich and Soibelman define 
\begin{equation}
\label{intdef}
\Phi_{Q,W}([X\xrightarrow{f}\XX^{\nilp}_{Q,W,\nn}]=\left(\int_X f^*\phi^{\rel}_{W_{\min}\boxplus \overline{Q}_{\tau_{Q,W}}}\right)\ehat_{\nn}\in\overline{\KK}^{\muu}(\Sta/\Spec(\Cp))[\LL^{1/2}][[\ehat_{\mm},\mm\in\NN^{\verts(Q)}]]
\end{equation}
where $\overline{Q}_{\tau_{Q,W}}$ is a function $\overline{Q}_{\tau_{Q,W}}(z)=\overline{Q}_{\tau_{Q,W}}(z,z)$ on $V$, for some pair of ind-constructible vector bundle $V$ on $\XX^{\nilp}_{Q,W,\nn}$ and nondegenerate inner product $\overline{Q}_{\tau_{Q,W}}$ on $V$ giving rise to the natural orientation data on $\XX^{\nilp}_{Q,W}$ arising from its realisation as the moduli space of objects in $\tw_0(\DD(Q,W))$.  That is, under the natural identification $\sDet(V)\cong \sDet(V)^*$ induced by $\overline{Q}_{\tau_{Q,W}}$, we obtain the natural orientation data $\tau_{Q,W}$ on $\XX^{\nilp}_{Q,W,\nn}$ given by (\ref{quadhome}) with its natural nondegenerate product.  The function, $W_{\min}$ is as defined after Proposition \ref{cmmt}.  The target is just the ring of formal power series in variables $\ehat_{\mm}$, with the usual\footnote{In fact usually one would twist the multiplication by some power of $-\LL^{1/2}$, but we escape this necessity as we only work with symmetric quivers.} multiplication $\ehat_{\mm'}\cdot\ehat_{\mm}=\ehat_{\mm'+\mm}$.  One extends to a map 
\[
\Phi_{Q,W}\colon \KK(\Sta/\XX^{\nilp}_{Q,W})\rightarrow \overline{\KK}^{\muu}(\Sta/\Spec(\Cp))[\LL^{1/2}][[\ehat_{\mm},\mm\in\NN^{\verts(Q)}]]
\]
by $\KK(\Sta/\Spec(\Cp))$-linearity and Proposition \ref{localis}.\medbreak
For general 3-Calabi--Yau categories the following is only a theorem if one is able to work with motivic vanishing cycles of formal functions and prove the motivic integral identity of \cite{KS}.  For the former, see the comment immediately following the theorem.  For the latter, see \cite{MotVan} or \cite{LQT12}.
\begin{thm}[\cite{KS}]
\label{intmappres}
The morphism $\Phi_{Q,W}\colon \KK(\Sta/\XX^{\nilp}_{Q,W})\rightarrow \overline{\KK}^{\muu}(\Sta/\Spec(\Cp))[\LL^{1/2}][[\ehat_{\nn}|\nn\in \NN^{\verts(Q)}]]$ is a $\KK(\Sta/\Spec(\Cp))$-algebra homomorphism.
\end{thm}
The issue with formal functions is not a serious one in our case.  Let $X/\Spec(\Cp)$ be a finite type scheme.  There is a multiplication $\star_X$ on $\KK^{\GG_m,n}(\Sta/\AA_{X}^1)$ given by 
\[
[Y_1\xrightarrow{f_1}\AA_{X}^1]\star_X [Y_2\xrightarrow{f_2}\AA_{X}^1]=[Y_1\times_X Y_2\xrightarrow{p}\AA_{X}^1],
\]
where the fibre product is with respect to the morphisms $\pi_X\circ f_1$ and $\pi_X\circ f_2$, and the map $p$ is defined by $(y_1,y_2)\mapsto (\pi_{\AA_{\Cp}^1}\circ f_1(y_1)+\pi_{\AA_{\Cp}^1}\circ f_2(y_2),\pi_X\circ f_1(y_1))$.  This multiplication descends to $\KK^{\muu}(\Sta/X)$.  We make the following definition.
\begin{prop}
\label{hack}
Let $V$ be a vector bundle on the scheme $X$, and let $f$ be a formal function on $V$.  Furthermore, assume that there exists a vector bundle $V'$ on $X$ and a quadratic form $\overline{Q}$ on $V'$, such that there is a formal change of coordinates on $V\oplus V'$ taking $f\boxplus \overline{Q}$ to a polynomial function $g$ (here we abuse notation and write $\overline{Q}(z)=\overline{Q}(z,z)$).  Then $\phi^{\rel}_f$ is well defined, and 
\[
\phi^{\rel}_f|_X=\phi^{\rel}_g|_X\star_X \phi^{\rel}_{\overline{Q}}|_X.
\]
\end{prop}
\begin{proof}
It is easy to show that 
\[
[\Lo_n(V'\oplus V')|_{X}\xrightarrow{(\pi\circ(\overline{Q}\boxplus\overline{Q})_*)\times p_n}\AA^1_X]=[X\xrightarrow{(0,\id_X)} \AA_X^1]=1\in\KK^{\GG_m,n}(\Sta/\AA_{X}^1)
\]
for all $n$.  It follows that there are equalities of relative motivic zeta functions 
\begin{align*}
Z_f(T)^{\mathrm{eq}}=&Z_{f\boxplus \overline{Q}\boxplus \overline{Q}}^{\mathrm{eq}}(T)\\
=&Z_{g\boxplus \overline{Q}}^{\mathrm{eq}}(T)
\end{align*}
and the result follows by the Thom--Sebastiani theorem.
\end{proof}
If one works with the minimal potentials of objects in the category of modules over a Ginzburg differential graded algebra for a quiver with polynomial potential, one only needs to deal with formal functions $f$ satisfying the conditions of Proposition \ref{hack}.
\medbreak
There is a more down to earth way to define the integration map for the Hall algebra $\KK(\Sta/\XX^{\nilp}_{Q,W})$.  In fact this second way extends without any effort to an integration map
\begin{equation}
\label{BBSmap}
\Phi_{\BBS,Q,W}\colon \KK(\Sta/\XX_{Q,W})\rightarrow \KK^{\muu}(\Sta/\Spec(\Cp))[\LL^{-1/2}][[\ehat_{\nn},\nn\in\NN^{\verts(Q)}]]
\end{equation}
exploited by Behrend, Bryan and Szendr\H{o}i to define and calculate motivic Donaldson--Thomas counts for Hilbert schemes of points on $\mathbb{C}^3$ in \cite{BBS}.  The Hodge theoretic version of this construction is a part of \cite{COHA}, see also \cite{DS09}.  One defines, similarly to the Kontsevich--Soibelman integration map,
\[
\Phi_{\BBS,Q,W}\colon [X\xrightarrow{f}\XX_{Q,W,\nn}]\mapsto\int_X f^*\phi_{\tr(W)}\ehat_{\nn}.
\]
Let 
\[
\overline{q}\colon \KK^{\muu}(\Sta/\Spec(\Cp))[\LL^{-1/2}][[\ehat_{\nn},\nn\in\NN^{\verts(Q)}]]\rightarrow \overline{\KK}^{\muu}(\Sta/\Spec(\Cp))[\LL^{-1/2}][[\ehat_{\nn},\nn\in\NN^{\verts(Q)}]]
\]
be the natural quotient map.  The following comparison theorem will be used in the proof of Propositions \ref{offdiag} and \ref{diag}.
\begin{prop}\cite[Thm.7.1.3]{Dav10a}
\label{comparisonthm}
There is an equality of maps $\overline{q}\circ\Phi_{\BBS,Q,W}|_{\KK(\Sta/\XX^{\nilp}_{Q,W})}=\Phi_{Q,W}$.
\end{prop}

\section{Motivic Donaldson--Thomas invariants of (-2)-curves}
\label{calcs}
\begin{figure}
\centering
\input{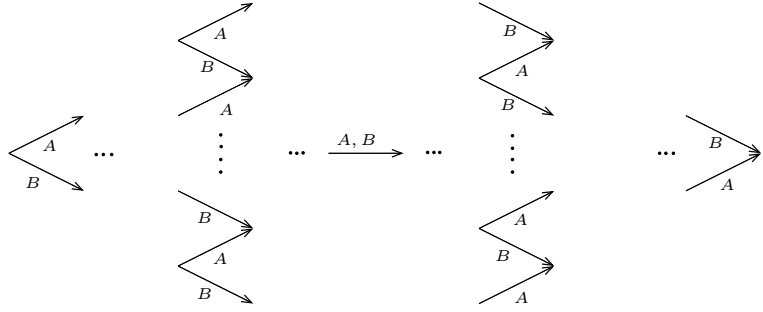}
\caption{Stable representations for $A_{Q_{\con}}$ (or $A_{Q_{\kron}}$) of dimension vector $(1,2), (n,n+1),\ldots,(n+1,n),(2,1)$.  The vertices represent a set of basis elements for the underlying vector space, while the labelled arrows represent the action of the homomorphism, labelled by those arrows, on this basis.  The $(1,1)$-dimensional representation in the centre of the figure lies in a family parametrised by $\mathbb{P}^1$.}
\label{somereps}
\end{figure}

\subsection{The calculation of the invariants}
We are finally able to calculate the motivic Donaldson--Thomas invariants of the category of nilpotent modules over $A_{Q_{-2},W_d}$, the noncommutative crepant resolution of $X_d$ defined in (\ref{Appear}) and explicitly described by (\ref{jacdefe}).  Firstly, we pick a stability condition, which for us is just an additive map
\[
\zeta\colon\NN^{\verts(Q_{-2})}\setminus 0\rightarrow \HH_+, 
\]
where $\HH_+$ is the set $\{r\cdot e^{i\theta}|r\in \mathbb{R}_{>0},\theta\in (0,\pi]\}$ and $\verts(Q_{-2})$ is the set of vertices of the quiver $Q_{-2}$ of Figure \ref{NCres}.  We make the genericity assumption that $\zeta$ does not map the whole of $\NN^{\verts(Q_{-2})}\setminus 0$ onto the same ray in $\Cp$.  As a result of the fact that $Q_{-2}$ is symmetric, the invariants we calculate will not depend on which stability condition $\zeta$ we pick.  We recall, regardless, that a module $M$ of slope $\theta:=\arg(\zeta(\bdim(M)))$ is called (semi)stable if for all proper submodules $N\subset M$, $\arg(\zeta(\bdim(N)))(\leqslant)\arg(\zeta(\bdim(M)))$, where the bracket denotes the fact that for semistability we only require the weak inequality.
\begin{lem}
\label{snilp}
Let $M$ be a semistable nilpotent $A_{Q_{-2},W_d}$-module with slope $\theta$.  Then $M$ is given by repeated extension by stable modules $M_{\alpha}$ of slope $\theta$, such that $M_{\alpha}\cdot (X+Y)=0$.
\end{lem}
\begin{proof}
The fact that $M$ admits a filtration with subquotients given by stable modules of slope $\theta$ is the statement of the existence of Jordan--H\"older filtrations.  For the second part, note that $X+Y\in A_{Q_{-2},W_d}$ is central, and so acts via module endomorphisms on all $A_{Q_{-2},W_d}$-modules.  This endomorphism is nilpotent for $M$, by assumption.  Define 
\[
F^mM=\Ker(\cdot(X+Y)^m\colon M\rightarrow M)
\]
to be the filtration of $M$ by the nilpotence degree of this endomorphism. then each $F^mM$ is semistable of slope $\theta$, since we have the short exact sequence
\[
0\rightarrow F^mM \rightarrow M\rightarrow \Image\big(\cdot(X+Y)^m\colon M\rightarrow M\big)\rightarrow 0
\]
where the middle term is semistable of slope $\theta$, and both the first and last terms have slope no greater than $\theta$, from which it follows that they have slope equal to $\theta$.  It follows that each subquotient $F^{m}M/F^{m-1}M$ is semistable of slope $\theta$, and the subquotients occurring in a refinement of the filtration $F^{\bullet}M$ to a Jordan--H\"older filtration of $M$ are all acted on by zero by $\cdot (X+Y)$, since each $F^{m}M/F^{m-1}M$ is.
\end{proof}
The data of a module over $A_{Q_{-2},W_d}$ is just the data of a module $M$ over $A_{Q_{\con},W_{\con}}$, the Jacobi algebra for the noncommutative conifold (see Remark \ref{conifold}), along with an endomorphism $\upsilon\colon M\rightarrow M$ given by the action of $X+Y$, satisfying 
\[
\upsilon^d=(a\mapsto a\cdot(AC+CA-BD-DB)).
\]
By Lemma \ref{snilp} and Remark \ref{conifold}, the semistable nilpotent modules of $A_{Q_{-2},W_{d}}$ are given by iterated extension of stable $A_{Q_{\con},W_{\con}}$-modules, considered as $A_{Q_{-2},W_d}$ modules by extension by zero.  The stable nilpotent modules for $A_{Q_{\con},W_{\con}}$ are classified in \cite[Thm.3.5]{NN}.  We have drawn a few of them in Figure \ref{somereps}.  There is one stable nilpotent module for each slope equal to $\zeta((n,n+1))$ or $\zeta((n+1,n))$, for $n\in \NN$ --- consider the vertices in Figure \ref{somereps} as a basis, then the arrows demonstrate the action of the morphisms assigned to $A$ and $B$ on this basis.  These stable modules have dimension vector $(n,n+1)$ or $(n+1,n)$ respectively.  For the slope $\zeta((1,1))$, the stable nilpotent modules are all of dimension vector $(1,1)$, and are parametrised by $\PP^1$.\medbreak
Recall from (\ref{nilpd}) that we denote by $\XX_{Q_{-2},W_d}^{\nilp}$ the substack of finite-dimensional $A_{Q_{-2},W_d}$-modules cut out by the equations $\tr(\rho(c))=0$ for all cyclic paths $c\in\mathbb{C}Q_{-2}$.  The isomorphism classes of closed points of this stack are in bijection with the isomorphism classes of nilpotent $A_{Q_{-2},W_d}$-modules.  We will use the familiar identity\footnote{This identity is just a fancy way of stating the existence and uniqueness of Harder--Narasimhan filtrations.} in the Hall algebra $\KK(\Sta/\XX_{Q_{-2},W_d}^{\nilp})$ defined in Section \ref{Hallalg}, where we abuse notation by omitting the obvious inclusion morphisms into $\XX^{\nilp}_{Q_{-2},W_d}$,
\begin{equation}
\label{HNfilt}
\prod_{\mbox{\scriptsize decreasing slope }\theta} [\XX^{\nilp,\zeta-\sss}_{Q_{-2},W_d,\theta}]=[\XX^{\nilp}_{Q_{-2},W_d}].
\end{equation}
Here $\XX^{\nilp,\zeta-\sss}_{Q_{-2},W_d,\theta}$ is the moduli stack of semistable nilpotent modules with slope $\theta$, and the product is the Hall algebra product defined by Kontsevich and Soibelman on $\KK(\Sta/\XX^{\nilp}_{Q_{-2},W_d})$ (see Section \ref{Hallalg}, and especially Remark \ref{Halldiff}).

\begin{prop}
\label{offdiag}
There is an equation of generating series in $\overline{\KK}^{\muu}(\Sta/\Spec(\Cp))[\LL^{1/2}][[\ehat_{\nn},\nn\in \NN^{\verts(Q_{-2})}]]$
\begin{equation}
\label{rigidpart}
\Phi_{Q_{-2},W_d}\left([\XX^{\nilp,\zeta-\sss}_{Q_{-2},W_d,\theta}]\right)=\Sym\left(\frac{\LL^{-1/2}(1-[\mu_{d+1}])}{\LL^{1/2}-\LL^{-1/2}}\ehat_{\nn}\right)
\end{equation}
for $\arg(\zeta(\nn))=\theta$, and $\nn=(n,n+1)$ or $\nn=(n+1,n)$ with $n\in \NN$.
\end{prop}
In Equation (\ref{rigidpart}) we consider $[\mu_{d+1}]$ as a $\mu_{d+1}$-equivariant motive via the action of multiplication.
\begin{proof}
The statement reduces to the computation of $W_{\min}$ and the orientation data above the unique $\nn$-dimensional semistable module, for each $\nn$ of slope $\theta$.  Note that on the geometric side of the derived equivalence (\ref{ncDE}) the unique stable module $M$ of slope $\theta$ is given by $\OO_C(a)[b]$ for some $a,b\in \ZZ$.  Since there is a derived autoequivalence of $\DDb(\Coh(Y_d))$ taking $\OO_C(a)[b]$ to $\OO_C$ we deduce that
\begin{equation*}
\dim(\Ext^1(M,M))=\begin{cases} 1 &\mbox{if } d\geq 2\\0&\mbox{otherwise,}\end{cases}
\end{equation*}
and in the case $d\geq 2$, $W_{\min}$ is given by $x^{d+1}$.  Both of these facts follow since the universal deformation of $\OO_C$ is over the Artinian ring $\Cp[x]/(x^{d})$.  The 3-Calabi--Yau category of semistable $A_{Q_{-2},W_d}$-modules of slope $\theta$ is, then, quasi-isomorphic as a cyclic $A_{\infty}$-category to the category $\tw_0(Q_L^1,X^{d+1})$, where for $a\in\NN$, $Q_L^a$ is the quiver with one vertex and $a$ loops.  
\medbreak
We claim that the orientation data $\tau_{Q_{-2},W_d}$ over $M$, considered as a point in $\XX^{\nilp}_{Q_{-2},W_d}$, is trivial if and only if $d\geq 2$.  For this, note that there are precisely two isomorphism classes of orientation data over a point; given two super line bundles $\VV_1$ and $\VV_2$ over a point, i.e.\ vector spaces with parity, and isomorphisms $\VV_i^{\otimes 2}\xrightarrow{\eta_i} \Cp$, there is an isomorphism $\VV_1\xrightarrow{f}\VV_2$ such that $\eta_2\circ(f\otimes f)=\eta_1$ if and only if the parity of $\VV_1$ is the same as that of $\VV_2$.  It is sufficient, then, to show that $\dim\left(\End^1_{\tw(\DD(Q_{-2},W_d))}(M)/\Ker(b_{\tw(\DD(Q_{-2},W_d)),1})\right)$ is even if and only if $d\geq 2$.  This follows from the following equations:
\begin{align}
\dim\left(\Hom^0_{\tw(\DD(Q_{-2},W_d))}(M,M)\right)\equiv &\nn(0)^2+\nn(1)^2\equiv 1&\mbox{(modulo }2)\\
\dim\left(\Hom^1_{\tw(\DD(Q_{-2},W_d))}(M,M)\right)\equiv&4\nn(0)\nn(1)+\nn(0)^2+\nn(1)^2\equiv 1&\mbox{(modulo }2)\\
\label{conred}\dim\left(\Ext^0_{\tw(\DD(Q_{-2},W_d))}(M,M)\right)\equiv&1&\mbox{(modulo }2).
\end{align}
The first two identities follow from the definitions of homomorphism spaces in $\tw(\DD(Q_{-2},W_d))$, and the identity (\ref{conred}) follows from the fact that $M$ is stable and hence simple.  We then calculate, modulo 2
\begin{align*}
1\equiv&\dim\left(\Hom^1_{\tw(\DD(Q_{-2},W_d))}(M,M)\right)\\
\equiv&\dim\left(\Image(b_{\tw(\DD(Q_{-2},W_d)),0})\right)+\dim\left(\Ext^1_{\tw(\DD(Q_{-2},W_d))}(M,M)\right)+\\&+\dim\left(\End^1_{\tw(\DD(Q_{-2},W_d))}(M)/\Ker(b_{\tw(\DD(Q_{-2},W_d)),1})\right) \\
\equiv&\dim\left(\Hom^0_{\tw(\DD(Q_{-2},W_d))}(M,M)\right)-\dim\left(\Ext^0_{\tw(\DD(Q_{-2},W_d))}(M,M)\right)+\\&+\dim\left(\Ext^1_{\tw(\DD(Q_{-2},W_d))}(M,M)\right)+\dim\left(\End^1_{\tw(\DD(Q_{-2},W_d))}(M)/\Ker(b_{\tw(\DD(Q_{-2},W_d)),1})\right)\\
\equiv&\dim\left(\Ext^1_{\tw(\DD(Q_{-2},W_d))}(M,M)\right)+\dim\left(\End^1_{\tw(\DD(Q_{-2},W_d))}(M)/\Ker(b_{\tw(\DD(Q_{-2},W_d)),1})\right).
\end{align*}
Thus, 
\[
\dim\left(\End^1_{\tw(\DD(Q_{-2},W_d))}(M)/\Ker(b_{\tw(\DD(Q_{-2},W_d)),1})\right)\equiv\begin{cases}0 &\textrm{if }d\geq 2\\1&\textrm{otherwise.}\end{cases}
\]
\medbreak
Now one can see directly that the orientation data assigned to the unique simple object $s_0$ of the category $\tw_0(\DD(Q_L^1,X^{d+1}))$ is trivial if and only if $d\geq 2$, since 
\[
b_{1}\colon \End_{\tw(\DD(Q_L^1))}^1(s_0)\rightarrow \End_{\tw(\DD(Q_L^1))}^2(s_0)
\]
is trivial if $d\geq 2$, otherwise this differential is an isomorphism.  So as well as having a cyclic $A_{\infty}$-isomorphism $\Xi$ from the subcategory of $\tw_0(Q_{-2},W_d)$ generated by $M$ under extensions to the category $\tw_0(Q_L^1,W^{d+1})$, we have an isomorphism of orientation data $\Xi^*(\tau_{Q_L^1,X^{d+1}})\cong\tau_{Q_{-2},W_d}$.  It follows that 
\begin{align*}
\Phi_{Q_{-2},W_d}\left([\XX^{\nilp,\zeta-\sss}_{Q_{-2},W_d,\theta}]\right)=&\Phi_{Q_L^1,X^{d+1}}\left([\XX^{\nilp}_{Q_L^1,X^{d+1}}]\right)_{\ehat_a\mapsto\ehat_{a\nn}}
\\=&\Phi_{\BBS,Q_L^1,X^{d+1}}\left([\XX^{\nilp}_{Q_L^1,X^{d+1}}]\right)_{\ehat_a\mapsto\ehat_{a\nn}}
\\=&\Phi_{\BBS,Q_L^1,X^{d+1}}\left([\XX_{Q_L^1,X^{d+1}}]\right)_{\ehat_a\mapsto\ehat_{a\nn}}
\end{align*}
where for the penultimate equation we used Proposition \ref{comparisonthm}, and for the final equation we use the fact that all finite-dimensional $A_{Q_L^1,X^{d+1}}$-modules are nilpotent.  The desired equality is then \cite[Thm.6.2]{DM11a}.
\end{proof}
\begin{prop}
\label{diag}
There is an equation of generating series in $\overline{\KK}^{\muu}(\Sta/\Spec(\Cp))[\LL^{1/2}][[\ehat_{\nn},\nn\in \NN^{\verts(Q_{-2})}]]$
\[
\Phi_{Q_{-2},W_d}\left([\XX_{Q_{-2},W_d,\theta}^{\nilp,\zeta-\sss}]\right)=\Sym\left( \sum_{n\geq 1}\frac{\LL^{-1/2}+\LL^{-3/2}}{\LL^{1/2}-\LL^{-1/2}}\ehat_{(n,n)}\right)
\]
for $\arg(\zeta((1,1)))=\theta$.
\end{prop}
\begin{proof}
The simple stable nilpotent modules $M$ with dimension vector $(1,1)$ are given by choosing two linear maps $M(A)$ and $M(B)$, from $\Cp$ to $\Cp$, not both equal to zero.  These modules correspond to the structure sheaves of points on the exceptional curve $C_d\subset Y_d$ under the derived equivalence (\ref{ncDE}).  Let $\VY_{n}^{\circ}$ be the subscheme of $\prod_{a\in \edge(Q_{-2})}\Hom(\Cp^{\nn(t(a))},\Cp^{\nn(s(a))})$, for $\nn=(n,n)$, the points of which satisfy the condition that the linear map assigned to $A$ is an isomorphism, and the Harder--Narasimhan filtration of the associated module only contains modules with dimension vector $(1,1)$.  The action of $\Gl_{\Cp}(\nn(1))$ on $\VY_n^{\circ}$ is free.  Taking the quotient by this action corresponds to forgetting the data of the isomorphism $A$, and identifying the two vertices of the quiver $Q_{-2}$, and so 
\[
\VY_{n}^{\circ}/(\Gl_{\Cp}(n)\times\Gl_{\Cp}(n))\cong \YY_{Q^5_L,n}, 
\]
where $Q^5_L$ is the five loop quiver, with loops labelled $B,C,D,X,Y$.  Furthermore, under the open inclusion $\YY_{Q_L^5,n}\hookrightarrow \YY_{Q_{-2},(n,n)}$, the function $\tr(W_d)$ pulls back to the function $\tr(W_d^{\circ})$, where $W_{d}^{\circ}$ is the superpotential 
\[
W_d^{\circ}=\frac{1}{d+1}X^{d+1}-\frac{1}{d+1}Y^{d+1}-XC+XDB+YC-YBD.
\]
If we define $\XX^{\nilp,\zeta-\sss,\circ}_{Q_{-2},W_d,\nn}$ to be the substack of $\XX_{Q_{-2},W_d,\nn}$, the points of which are $\zeta$-semistable nilpotent $A_{Q_{-2},W_d}$-modules $M$ such that $\theta(A)$ is an isomorphism, then we have shown that $\XX^{\nilp,\zeta-\sss,\circ}_{Q_{-2},W_d,\nn}$ is naturally a substack of $\XX_{Q^5_L,W^{\circ}_d,n}\subset \YY_{Q^5_L,n}$, and we identify it as the stack of $n$-dimensional representations for the Jacobi algebra associated to $(Q^5_L,W_d^{\circ})$ such that all loops apart from $B$ act via nilpotent linear maps.  We denote
\begin{equation}
\label{xodef}
\XX^{\nilp,\zeta-\sss,\circ}_{Q_{-2},W_d}=\coprod_{n\in\mathbb{N}}\XX^{\nilp,\zeta-\sss,\circ}_{Q_{-2},W_d,(n,n)}.
\end{equation}
Under the derived equivalence (\ref{ncDE}) the stack (\ref{xodef}) is the substack of coherent sheaves supported on the exceptional curve $C_d\subset Y_d$, away from a fixed point.  We will denote this point by $p$.  
\medbreak
Denote by $\XX^{\nilp,\zeta-\sss,p}_{Q_{-2},W_d}$ the stack of modules for $A_{Q_{-2},W_d}$ which are supported at the point $p$ under the derived equivalence (\ref{ncDE}).  Then since every sheaf that is scheme-theoretically supported on $C_d$ with zero-dimensional support splits uniquely as a direct sum of a coherent sheaf supported at $p$ and a coherent sheaf supported away from $p$, there is an identity in the motivic Hall algebra $(\KK(\Sta/\XX_{Q_{-2},W_d}),\star_{\KS})$
\begin{equation}
\label{obvsplit}
[\XX^{\nilp,\zeta-\sss}_{Q_{-2},W_d,\theta}]=[\XX^{\nilp,\zeta-\sss,\circ}_{Q_{-2},W_d}]\star_{\KS} [\XX^{\nilp,\zeta-\sss,p}_{Q_{-2},W_d}].
\end{equation}
Now note that there is a splitting
\begin{equation*}
W_d^{\circ}=XDB-XBD+(X-Y)(BD-C+\frac{1}{d+1}(X^{d}+X^{d-1}Y+\ldots+Y^d)).
\end{equation*}
We deduce that after giving $Y_{Q^5_L,n}$ the coordinates $X,D,B, Y'=X-Y,C'=BD-C+\frac{1}{d+1}(X^{d}+X^{d-1}Y+\ldots+Y^d)$, we have
\begin{align}
\label{eq0} \Phi_{Q_{-2},W_d}(\XX^{\nilp,\zeta-\sss,\circ}_{Q_{-2},W_d})=&\sum_{n\geq 0}\left(\int_{\XX^{\nilp,\zeta-\sss,\circ}_{Q_{-2},W_d,(n,n)}\subset \YY_{Q_{-2},(n,n)}}\phi_{\tr(W_{d})}\right) \ehat_{(n,n)}
\\=& \sum_{n\geq 0}[\Gl_{\Cp}(n)]^{-1}\LL^{n^2/2}\left(\int_{\{X,Y',C' \mbox{ and }D \mbox{ nilpotent}\}\subset \VY_{Q^5_L,n}}
\phi_{\tr(XDB-XBD)\boxplus \tr(Y'C')}\right)\ehat_{(n,n)}\nonumber
\\ \label{eq1} =&\sum_{n\geq 0}[\Gl_{\Cp}(n)]^{-1}\LL^{n^2/2}\left(\int_{\{X \mbox{ and }D \mbox{ nilpotent}\}\subset \VY_{Q^3_L,n}} \phi_{\tr(XDB-XBD)}\right)\ehat_{(n,n)}.
\end{align}

Here (\ref{eq0}) comes from the comparison theorem (Theorem \ref{comparisonthm}), and (\ref{eq1}) comes from applying the motivic Thom-Sebastiani theorem.  Now, giving the coordinates $X,D,B$ weights $0$, $0$ and $1$ respectively, and applying the weight one version of Conjecture \ref{conjecture} (which is a theorem), with $Z'$ the scheme of pairs of matrices labelled by $X$ and $D$, we obtain
\begin{align}
\Phi_{Q_{-2},W_d}(\XX^{\nilp,\zeta-\sss,\circ}_{Q_{-2},W_d})=&\sum_{n\geq 0}[\Gl_{\Cp}(n)]^{-1}\LL^{-n^2} (\{X\mbox{ and } D\mbox{ nilpotent, }XD\neq DX\} (\LL^{n^2-1}-\LL^{n^2-1})-\\&-\{X\mbox{ and } D\mbox{ nilpotent, }XD=DX\}\LL^{n^2})\ehat_{(n,n)}\nonumber
\\  = &\sum_{n\geq 0}[\Gl_{\Cp}(n)]^{-1}\Comm_{n,\nilp}\ehat_{(n,n)}\nonumber
\\=&\sum_{n\geq 0}\CComm_{n,\nilp}\ehat_{(n,n)}\nonumber
\\ \label{eq3} =&\left(\sum_{n\geq 0}\CComm_n\ehat_{(n,n)}\right)^{-\LL^2}
\\ \label{eq4} =&\Sym\left(\sum_{n\geq 1} \frac{\LL^{-1/2}}{\LL^{1/2}-\LL^{-1/2}}\ehat_{(n,n)}\right)
\end{align}
where $\Comm_n$ is the variety of pairs of commuting $n\times n$ matrices, and $\CComm_n$ is its quotient under the conjugation action of $\Gl_{\Cp}(n)$.  One can think of this stack as the stack of length $n$ coherent sheaves on $\Cp^2$.  Above, $\Comm_{n, \nilp}$ and $\CComm_{n, \nilp}$ are the variety and stack, respectively, of nilpotent commuting matrices, the second of which one should think of as the stack of coherent sheaves on $\Cp^2$ scheme-theoretically supported at the origin.  Then (\ref{eq3}) follows from the definition of the power structure in Section \ref{lambdarings}, and (\ref{eq4}) follows from the main result of \cite{FF60}, as in \cite[Prop.1.1]{BBS}.
Similarly one deduces that 
\begin{align*}
\Phi_{Q_{-2},W_d}(\XX^{\nilp,\zeta-\sss,p}_{Q_{-2},W_d})=\Sym\left(\sum_{n\geq 1} \frac{\LL^{-3/2}}{\LL^{1/2}-\LL^{-1/2}}\ehat_{(n,n)}\right)
\end{align*}
and now the result follows from applying the integration map to the Hall algebra identity (\ref{obvsplit}).
\end{proof}
The following theorem now follows from applying the integration map to the Harder--Narasimhan identity (\ref{HNfilt}) in $\KK(\Sta/\XX^{\nilp}_{Q_{-2},W_d})$.
\begin{thm}
\label{mainthm}
There is an equality in $\overline{\KK}^{\muu}(\Sta/\Spec(\Cp))[\LL^{1/2}][[\ehat_{\nn},\nn\in \NN^{\verts(Q_{-2})}]]$:
\begin{equation*}
\Phi_{Q_{-2},W_d}(\XX^{\nilp}_{Q_{-2},W_d})=\Sym\left(\sum_{n\geq 0}\frac{\LL^{-1/2}(1-[\mu_{d+1}])}{\LL^{1/2}-\LL^{-1/2}}(\ehat_{(n,n+1)}+\ehat_{(n+1,n)})+\sum_{n\geq 1}\frac{\LL^{-1/2}+\LL^{-3/2}}{\LL^{1/2}-\LL^{-1/2}}\ehat_{(n,n)}\right).
\end{equation*}
In particular, the motivic Donaldson--Thomas invariants $\Omega_{\zeta}^{\nilp}$ counting nilpotent $A_{Q_{-2},W_d}$-modules (for any $\zeta$) are given by 
\begin{equation*}
\Omega^{\nilp}_{\zeta}(\nn)=\begin{cases} (1-[\mu_{d+1}])\cdot\LL^{-1/2}&\mbox{if there exists }n\in\NN \mbox{ such}\\ &\mbox{that }\nn=(n,n+1)\mbox{ or }\nn=(n+1,n),\\ \PP^1\cdot \LL^{-3/2} &\mbox{if there exists }n\in\NN\mbox{ such that }\nn=(n,n).
\end{cases}
\end{equation*}
\end{thm}

\subsection{Calculation using equivariant vanishing cycles}
We repeat the above calculations, but this time the other side of the comparison theorem (Proposition \ref{comparisonthm}).  It is more natural there to work out the Donaldson--Thomas invariants for the category of finite-dimensional $A_{Q_{-2},W_d}$-modules, not just the nilpotent ones.  First we recall the following conjecture from \cite{DM11a}.
\begin{conj}
\label{conjecture}
Let $Z'$ be a smooth scheme with trivial $\GG_m$-action, and let $\AA_{\Cp}^n$ carry a $\GG_m$-action with nonnegative weights.  Let $Z=Z'\times \AA_{\Cp}^n$ with the induced $\GG_m$-action, and let $f\colon Z\rightarrow \AA_{\Cp}^1$ be a $\GG_m$-equivariant function, with $\GG_m$ acting on the target with weight $s>0$.  Then there is an equality in $\KK^{\muu}(\Sta/Z')$
\begin{equation}
\label{conjid}
q_*\phi_f=\LL^{-\dim(Z)/2}([f^{-1}(0)]-[f^{-1}(1)]),
\end{equation}
where $f^{-1}(0)$ carries the trivial $\muu$-action, and the $\muu$-action on $f^{-1}(1)$ is given by the natural $\mu_s$-action, and both are considered as varieties over $Z'$ via the projection $q\colon Z\rightarrow Z'$.  Equivalently, there is an identity in $\lim_{n\rightarrow\infty}\KK^{\GG_m,n}(\Sta/\AA^1_{Z'})$
\begin{equation}
\label{conjid2}
q_*\phi_f=\LL^{-\dim(Z)/2}[Z\xrightarrow{f\times q} \AA_{Z'}^1].
\end{equation}
\end{conj}
This conjecture follows from the proof of Theorem 5.9 in \cite{DM11a}, under the assumption that the weights on $\AA_{\Cp}^n$ are all less than or equal to one.  If, in addition, $s=1$, then \cite[Thm.5.9]{DM11a} is a result of \cite[Prop.1.11]{BBS}.  While this paper was being prepared for publication, we were informed that Johannes Nicaise and Sam Payne have a strategy for proving the general case based on tropical geometry and Hrushovski-Kazhdan motivic integration.
\begin{thm}
For $d\leq 2$, there is an identity 
\[
\Phi_{\BBS,Q_{-2},W_d}\left([\XX_{Q_{-2},W_d}]\right)=\Sym\left(\sum_{n\geq 0}\frac{\LL^{-1/2}(1-[\mu_{d+1}])}{\LL^{1/2}-\LL^{-1/2}}(\ehat_{(n,n+1)}+\ehat_{(n+1,n)})+\sum_{n\geq 1}\frac{\LL^{3/2}+\LL^{1/2}}{\LL^{1/2}-\LL^{-1/2}}\ehat_{(n,n)}\right)
\]
in $\KK^{\muu}(\Sta/\Spec(\Cp))[\LL^{1/2}][[\ehat_{\nn},\nn\in \NN^{\verts(Q_{-2})}]]$.  Assuming Conjecture \ref{conjecture} this identity holds for all $d$.  It follows that the motivic Donaldson--Thomas invariants $\Omega_{\zeta}(\nn)$ (which do not depend on $\zeta$) are given by
\begin{equation*}
\Omega_{\zeta}(\nn)=\begin{cases} (1-[\mu_{d+1}])\cdot\LL^{-1/2}&\mbox{if there exists }n\in\NN \mbox{ such}\\ &\mbox{that }\nn=(n,n+1)\mbox{ or }\nn=(n+1,n),\\ [Y_d]_{\virt}:=\LL^{\frac{-\dim(Y_d)}{2}}\cdot [Y_d] &\mbox{if there exists }n\in\NN\mbox{ such that }\nn=(n,n).
\end{cases}
\end{equation*}
\end{thm}
The explicit description given in Section 2 shows that $Y_d$ is a Zariski locally trivial fibre bundle over the exceptional curve $C_d\cong \mathbb{P}^1_{\Cp}$ with fibre $\AA^2_{\Cp}$, and so 
\begin{align*}
[Y_d]=&(\LL^1+1)\LL^2\\
[Y_d]_{\virt}=&\LL^{3/2}+\LL^{1/2}
\end{align*}
in $\KK^{\muhat}(\Sta/\Spec(\Cp))$. The transition functions are linear only for $d=1$.  In particular, although $Y_{d}\ncong Y_{d'}$ for $d\neq d'$, the classes $[Y_d]$ and $[Y_d]_{\virt}$ do not depend on $d$.
\begin{proof}
For $\beta=(\beta_n,\ldots,\beta_1)$ a path in a quiver $Q$, and for
\[
M\in \VY_{Q,\nn}:=\prod_{a\in \edge(Q)}\Hom(\Cp^{\nn(t(a))},\Cp^{\nn(s(a))})
\]
we write $M(\beta)=M(\beta_n)\circ\ldots\circ M(\beta_1)$.  We let $\GG_m$ act on $\VY_{Q_{-2},\nn}$ via 
\[
(z\cdot M)(E)=z^{\iota(E)}\cdot M(E),
\]
where $\iota(E)=1$ if $E=X,Y,A,B$, and $\iota(E)=d-1$ if $E$=$C,D$.  Then 
\[
\tr(W_d)\colon Y_{Q_{-2},\nn}\rightarrow\AA_{\Cp}^1
\]
is $\GG_m$-equivariant, after giving $\AA_{\Cp}^1$ the weight $(d+1)$-action.  It follows from our assumption $d\leq 2$, or from Conjecture \ref{conjecture}, for general $d$, that 
\begin{equation}
\label{c5u}
\int_{\XX_{Q_{-2},W_d},\nn} \phi_{\tr(W_d)}=\int_{\YY_{Q_{-2}},\nn} \phi_{\tr(W_d)}=[\VY_{Q_{-2},\nn}\xrightarrow{\tr(W_d)}\AA_{\Cp}^1]\cdot\LL^{-2\nn(0)\nn(1)}\cdot [\Gl_{\Cp}(\nn(0))\times\Gl_{\Cp}(\nn(1))]^{-1}
\end{equation}
in $\KK^{\GG_m,d+1}(\Sta/\AA_{\Cp}^1)$ (in fact all subsequent calculations will take place in this ring).  The first of these equalities follows from the fact that the motive $\phi_{\tr(W_d)}$ is supported on the critical locus of $\tr(W_d)$, which is just $\XX_{Q_{-2},W_d,\nn}$.\medbreak
For a set of edges $E'\subset E(Q)$ let $Q\setminus E'$ be the quiver obtained by deleting the edges of $E'$ (this quiver has the same vertex set as $Q$).  If $W$ is a potential on $\mathbb{C} Q$, we denote by $W\setminus E'$ the potential on $Q\setminus E'$ obtained by changing the coefficient of any term in $W$ containing any edge of $E'$ to zero.  By abuse of notation we will often denote the potential $W\setminus E'$ on $Q\setminus E'$ by $W$.  There is a natural projection 
\[
\pi_C\colon \VY_{Q_{-2},\nn}\rightarrow \VY_{Q_{-2}\setminus \{C\},\nn}
\]
given by forgetting the data $M(C)$.  We consider this as the projection from the total space of a rank $\nn(0)\cdot\nn(1)$ vector bundle which we denote $\tilde{C}$.  There is an obvious equality
\[
[\VY_{Q_{-2},\nn}\xrightarrow{\tr(W_d)} \AA_{\Cp}^1]=[\pi_C^{-1} \VY_{Q_{-2}\setminus \{C\},\nn,M(AX)=M(YA)}\xrightarrow{\tr(W_d)}\AA_{\Cp}^1]+[\pi_C^{-1} \VY_{Q_{-2}\setminus \{C\},\nn,M(AX)\neq M(YA)}\xrightarrow{\tr(W_d)}\AA_{\Cp}^1].
\]
The restriction of the vector bundle $\tilde{C}$ to $\VY_{Q_{-2}\setminus \{C\},\nn,M(AX)\neq M(YA)}$ has a rank $(\nn(0)\cdot\nn(1)-1)$ sub-bundle $\tilde{C}_0$, given by those choices of $M(C)$ such that $\tr(M(CAX)-M(CYA))=0$.  The action of $\GG_m$ on $\VY_{Q_{-2}\setminus \{C\},\nn,M(AX)\neq M(YA)}$ is free, and from the corresponding non-equivariant statement on the quotient we deduce that after $\GG_m$-equivariant constructible decomposition of the base $\VY_{Q_{-2}\setminus \{C\},\nn,M(AX)\neq M(YA)}$, the inclusion $\tilde{C}_0\subset \tilde{C}|_{\VY_{Q_{-2}\setminus \{C\},\nn,M(AX)\neq M(YA)}}$ splits, and we may write
\begin{equation}
\label{cdec}
[\pi_C^{-1} \VY_{Q_{-2}\setminus \{C\},\nn,M(AX)\neq M(YA)}\xrightarrow{\tr(W_d)}\AA_{\Cp}^1]=[\tilde{C}_0\times\AA_{\Cp}^1\xrightarrow{\tr(W_d\setminus C) + \pi_{\AA_{\Cp}^1}}\AA_{\Cp}^1]
\end{equation}
where we have abused notation by identifying $\tilde{C}_0$ with its constructible decomposition.  After a change of coordinates we may write the right hand side of (\ref{cdec}) as $[\tilde{C}_0\times\AA_{\Cp}^1\xrightarrow{\pi_{\AA_{\Cp}^1}}\AA_{\Cp}^1]$.  Clearly this belongs to $\mathfrak{I}_{d+1}$.  We deduce that 
\begin{align*}
[\VY_{Q_{-2},\nn}\xrightarrow{\tr(W_d)} \AA_{\Cp}^1]=&[\pi_C^{-1}\VY_{Q_{-2}\setminus \{C\},\nn,M(AX)=M(YA)}\xrightarrow{\tr(W_d)}\AA_{\Cp}^1]
\\=&\LL^{\nn(0)\cdot\nn(1)}\cdot [\VY_{Q_{-2}\setminus \{C\},\nn,M(AX)=M(YA)}\xrightarrow{\tr(W_d\setminus C)}\AA_{\Cp}^1]
\end{align*}
and similarly
\begin{align}
\label{kronappear}
[\VY_{Q_{-2},\nn}\xrightarrow{\tr(W_d)} \AA_{\Cp}^1]=&\LL^{2\nn(0)\nn(1)}\cdot [E_{Q_{\kron},\nn}\xrightarrow{\tr(X^{d+1}-Y^{d+1})}\AA_{\Cp}^1].
\end{align}
where we define
\begin{align*}
E_{Q_{\kron},\nn}:=&(\VY_{Q_{-2}\setminus \{C,D\},\nn})|_{\substack{M(AX)=M(YA)\\M(BX)=M(YB)}}
\end{align*}
and define the stacks 
\begin{align*}
\EE_{Q_{\kron},\nn}:=&E_{Q_{\kron},\nn}/\left(\Gl_{\Cp}(\nn(0))\times \Gl_{\Cp}(\nn(1))\right)\\
\EE_{Q_{\kron}}:=&\coprod_{\nn\in\NN^2}\EE_{Q_{\kron},\nn}.
\end{align*}
These stacks represent pairs $(M,\xi)$, where $M$ is a right $A_{Q_{\kron}}$-module, and $\xi=X + Y$ is an endomorphism of $M$, where $Q_{\kron}$ is the Kronecker quiver with two vertices $x_0$ and $x_1$, and two arrows $A,B$, both going from $x_0$ to $x_1$.   In other words, $\EE_{Q_{\kron}}$ is the stack of finite-dimensional $\Bkron:=A_{Q_{\kron}}[z]$-modules.  By Beilinson's theorem $\Der(\rmod A_{Q_{\kron}})$ is derived equivalent to the category of coherent sheaves on $\mathbb{P}^1$ via the derived equivalence $\RHom(E,-)$, where 
\[
E=\mathcal{O}_{\mathbb{P}^1}\oplus \mathcal{O}_{\mathbb{P}^1}(1).
\]
Similarly, $\Der(\rmod\Bkron)$ is derived equivalent to the category of coherent sheaves on $\Tot(\mathcal{O}_{\mathbb{P}^1})$ via the derived equivalence $\RHom(\pi^*E,-)$, where 
\[
\pi\colon \Tot(\mathcal{O}_{\mathbb{P}^1})\rightarrow \mathbb{P}^1
\]
is the projection.  We claim that for $M_{\alpha}$ and $M_{\beta}$ two semistable $\Bkron$-modules with the slope of $M_{\alpha}$ lower than that of $M_{\beta}$, the following group vanishes
\[
\Ext^2_{\rmod \Bkron}(M_{\alpha},M_{\beta})=0.
\]
By the five lemma and the existence of Jordan--H\"older filtrations it is enough to prove the claim in the case in which both $M_{\alpha}$ and $M_{\beta}$ are stable.  By Serre duality, and the above derived equivalence, this is equivalent to showing that 
\[
\Hom(\mathcal{F}_1,\mathcal{F}_2(-2))=0
\]
for $\mathcal{F}_1$ occurring before $\mathcal{F}_2$ in the ordered collection of objects of $\Der(\Coh(\mathbb{P}^1))$:
\[
\mathcal{O}_{\mathbb{P}^1}(-1)[1],\mathcal{O}_{\mathbb{P}^1}(-2)[1],\ldots, \mathcal{O}_{\pt},\ldots, \mathcal{O}_{\mathbb{P}^1}(1),\mathcal{O}_{\mathbb{P}^1}
\]
which is clear.

Let $\aleph_{\zeta}$ be the set of all possible Harder--Narasimhan types of finite-dimensional $\Bkron$-modules.  We could equally have defined $\aleph_{\zeta}$ as the set of all possible Harder--Narasimhan types of $A_{Q_{\kron}}$-modules, since the endomorphism $z$ in the definition of $\Bkron$ preserves the Harder--Narasimhan filtration of the underlying $A_{Q_{\kron}}$-module.  Let $\gamma=(\nn^1,\ldots,\nn^s)\in\aleph_{\zeta}$, let
\[
\EE_{Q_{\kron},\gamma}
\]
be the stack of $\Bkron$-modules of Harder--Narasimhan type $\gamma$, and let 
\[
\JH:\EE_{Q_{\kron},\gamma}\rightarrow \prod_{i\leq s}\EE^{\zeta-\sss}_{Q_{\kron},\nn^i}
\]
be the map taking a module to its Jordan--H\"older filtration.  Then above a module $M=M_1\oplus\ldots \oplus M_s$, the fibre of $\JH$ is given by a stack $[\AA^{m}/\AA^n]$, where
\begin{align*}
n=&\sum_{1\leq i<j\leq s}\dim(\Hom(M_j,M_i))\\
m=&\sum_{1\leq i<j\leq s}\Ext^1(\Hom(M_j,M_i)).
\end{align*}
On the other hand, by the vanishing of $\Ext^2$s, and the fact that the Euler form on $\Coh_{\mathrm{cpct}}(\Tot(\mathcal{O}_{\mathbb{P}^1}))$ vanishes, each of the differences $\dim(\Hom(M_j,M_i))-\Ext^1(\Hom(M_j,M_i))$ vanishes, and so $n=m$.  We deduce from relation (\ref{rel1}) that
\begin{align}
\label{fixprop}
[\EE_{Q_{\kron},\gamma}\xrightarrow{\tr(W_d)}\AA_{\Cp}^1]=\prod_{0\leq i\leq s}[\EE_{Q_{\kron},\nn^i}^{\zeta-\sss}\xrightarrow{\tr(W_d)}\AA_{\Cp}^1].
\end{align}
If $\nn^i$ is equal to $a\cdot(n,n\pm 1)$ then 
\begin{align*}
[\EE^{\zeta-\sss}_{Q_{\kron},\nn^i}\xrightarrow{\tr(W_d)}\AA_{\Cp}^1]=&[\Mat_{a\times a}(\Cp)/\Gl_{\Cp}(a)\xrightarrow{\tr(nX^{d+1}-(n\pm 1)X^{d+1})}\AA_{\Cp}^1]
\\=&[\Mat_{a\times a}(\Cp)/\Gl_{\Cp}(a)\xrightarrow{\tr(X^{d+1})}\AA_{\Cp}^1].
\end{align*}
Similarly, if $\nn^i=a\cdot(1,1)$, then the function $\tr(W_d)$ is zero, restricted to $\EE^{\zeta-\sss}_{Q_{\kron},\nn^i}$.  This stack is just the stack of length $a$ coherent zero-dimensional sheaves on $\PP^1$ with an endomorphism.  It follows that
\[
\sum_{a\geq 0}[\EE^{\zeta-\sss}_{Q_{\kron},((a,a))}\rightarrow \AA_{\Cp}^1]\ehat_{(a,a)}=(\sum_{a\geq 0}[\EE^{\mathbb{C}}_{Q_{\kron},((a,a))}\rightarrow \AA_{\Cp}^1]\ehat_{(a,a)})^{\PP^1},
\]
where $\EE^{\mathbb{C}}_{Q_{\kron},((a,a))}$ is the stack of pairs $(M,\xi)$, where $M$ is a coherent $\OO_{\PP^1}$-module supported at zero, and $\xi$ is an endomorphism of $M$.  This is just the stack of pairs of commuting matrices $N_1$ and $N_2$ such that $N_1$ is nilpotent, which in turn is the stack of coherent sheaves on $\mathbb{C}^2$ scheme-theoretically supported on zero dimensional subschemes of a fixed coordinate line.  As in \cite[Sec.5.6]{COHA} one deduces that 
\begin{equation}
\sum_{a\geq 0}[\EE^{\zeta-\sss}_{Q_{\kron},((a,a))}\rightarrow \AA_{\Cp}^1]\ehat_{(a,a)}=\Sym\left(\sum_{n\geq 1}\frac{\LL^{3/2}+\LL^{1/2}}{\LL^{1/2}-\LL^{-1/2}}\ehat_{(n,n)}\right).
\end{equation}
Finally, putting all this together, we have
\begin{align*}
\Phi_{\BBS,Q_{-2},W_d}\left([\XX_{Q_{-2},W_d}]\right)=&\sum_{\nn\in\NN^{\verts(Q_{-2})}}\int_{\YY_{Q_{-2},\nn}}\phi_{\tr(W_d)}\ehat_{\nn}
\\=&\sum_{\nn\in\NN^{\verts(Q_{-2})}}\LL^{-2\nn(0)\cdot\nn(1)}[\VY_{Q_{-2},\nn}\xrightarrow{\tr(W_d)}\AA_{\Cp}^1]\cdot [\Gl_{\Cp}(\nn(0))\times\Gl_{\Cp}(\nn(1))]^{-1}\ehat_{\nn}
\\=&\sum_{\nn\in\NN^{\verts(Q_{-2})}}[\EE_{Q_{\kron},\nn}\xrightarrow{\tr(W_d)}\AA_{\Cp}^1]\ehat_{\nn}
\\=&\sum_{\gamma\in\aleph_{\zeta}}[\EE_{Q_{\kron},\gamma}\xrightarrow{\tr(W_d)}\AA_{\Cp}^1]\ehat_{\nn}
\\=&\prod_{\nn=(n,n\pm1)}\left(\sum_{a\geq 0}[\Mat_{a\times a}(\Cp)/\Gl_{\Cp}(a)\xrightarrow{\tr(X^{d+1})}\AA_{\Cp}^1]\ehat_{a \nn}\right)\cdot \Sym\left(\sum_{n\geq 1}\frac{\LL^{3/2}+\LL^{1/2}}{\LL^{1/2}-\LL^{-1/2}}\ehat_{(n,n)}\right)
\\=&\Sym\left(\sum_{\nn=(n,n\pm 1)}[\AA_{\Cp}^1\xrightarrow{z\mapsto z^{d+1}}\AA_{\Cp}^1]\cdot \LL^{-1/2}(\LL^{1/2}-\LL^{-1/2})^{-1}\ehat_{\nn}\right)\cdot\\ \cdot&\Sym\left(\sum_{n\geq 1} \frac{\LL^{3/2}+\LL^{1/2}}{\LL^{1/2}-\LL^{-1/2}}\ehat_{(n,n)}\right),
\end{align*}
where for the final equality we have again used the calculation of the motivic Donaldson--Thomas invariants for the one loop quiver with homogeneous potential from \cite[Thm.6.2]{DM11a}, and we are done.
\end{proof}

\begin{rem}
It is possible to give the category of (not necessarily nilpotent) $A_{Q_{-2},W_d}$-modules the structure of a cyclic $A_{\infty}$-category, and prove the above result in this framework, for arbitrary $d$.
\end{rem}
\begin{rem}
\label{conremark}
We return to the case $d=1$ and the comparison with the calculations of \cite{conifold, MMNS11}.  There, an alternative quiver with potential $(Q_{\con},W_{\KW})$ is used to give a Jacobi algebra presentation of the noncommutative resolution of $X_d$, where $Q_{\con}$ is the subquiver of $Q_{-2}$ depicted in Figure \ref{KWquiver} and 
\[
W_{\KW}=CADB-CBDA.
\]
Then it is not hard to see that the inclusion of algebras $\mathbb{C} Q_{\con}\hookrightarrow \mathbb{C} Q_{-2}$ induced by the inclusion of quivers $Q_{\con}\subset Q_{-2}$ induces an isomorphism of algebras $A_{Q_{\con},W_{\KW}}\cong A_{Q_{-2},W_1}$, and so we have \textit{two} derived equivalences
\[
\xymatrix{
\DDb(\rmod A_{Q_{-2},W_1})\ar[r]^-{\Delta_1}&\DDb(\Coh(Y_d))&\DDb(\rmod A_{Q_{\con},W_{\KW}})\ar[l]_-{\Delta_2}.
}
\]
Consider the sheaf $\mathcal{O}_{C_1}$, via either derived equivalence it corresponds to the unique simple object of dimension vector $(1,0)$.  The motivic DT invariant for this dimension vector, calculated in the category $\rmod A_{Q_{-2},W_1}$, with the natural orientation data coming from the presentation of this algebra as a Jacobi algebra, is $(1-[\mu_2])\LL^{-1/2}$. On the other hand, the motivic DT invariant for this dimension vector, calculated in the category $A_{Q_{\con},W_{\KW}}$, is given by calculating the motivic vanishing cycle of the function $\tr(W_{\KW})$ on the space $\pt$ --- the space parametrising $(1,0)$-dimensional representations of the quiver $Q_{\con}$.  This function is zero, and so we have $[\phi_{\tr(W_{\KW})|_{\pt}}]=[\pt]=1\neq (1-[\mu_2])\LL^{-1/2}$.  The difference is accounted for by the difference in natural orientation data coming from the two presentations of the noncommutative resolution as a Jacobi algebra.
\end{rem}
\medbreak
\begin{figure}
\centering
\input{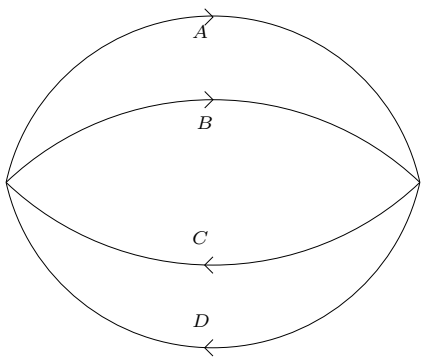}
\caption{The quiver $Q_{\con}$.}
\label{KWquiver}
\end{figure}

\begin{rem}
Restricting the derived equivalence (\ref{ncDE}), there are derived equivalences
\begin{align*}
\Dern(\rmod A_{Q_{-2},W_d})\cong&\Derbe(\Coh(Y_d)),\\
\Derf(\rmod A_{Q_{-2},W_d})\cong&\Dercpct(\Coh(Y_d))
\end{align*}
where $\Dercpct(\Coh(Y_d))$ is the subcategory of the bounded derived category of coherent sheaves on $Y_d$ with compactly supported total cohomology, while $\Derbe(\Coh(Y_d))$ is the subcategory of coherent sheaves with set-theoretic support contained in the exceptional curve $C_d$.  This is the explanation for the fact that 
\[
\Omega^{\nilp}_{\zeta}(\nn)=\Omega_{\zeta}(\nn)
\]
for all $\nn$ not counting point sheaves under the derived equivalence (\ref{ncDE}), as $C_d$ is the only proper subvariety of $Y_d$ of dimension greater than zero.
\end{rem}
\begin{rem}
Let $\nn=(m,n)$ with $m,n\in\mathbb{N}$ and $m=n\pm 1$.  Then the numerical DT invariant $\omega_{\zeta}(\nn)$ is extracted from our motivic DT invariant by first taking the Hodge spectrum 
\begin{align*}
\SP(\Omega_{\zeta}(\nn))=&(\sum_{l=1,d}u^{\frac{l}{d+1}}v^{\frac{d+1-l}{d+1}})u^{-1/2}v^{-1/2}\\
=&\sum_{l=1,d}u^{\frac{2l-d-1}{2d+2}}v^{\frac{d+1-2l}{2d+2}}
\end{align*}
and then replacing $u$ and $v$ with $q$ and setting $q^{1/2}=1$.  The Hodge spectrum is as defined and discussed in \cite[Sec.4.3]{KS}; in general for a $\mu_{d+1}$ equivariant variety $X$, the expression $\SP([X])$ is given by taking the usual mixed Hodge polynomial of the compactly supported cohomology of $X$, and multiplying the summand with eigenvalue $\exp(\alpha i/(d+1))$ under the monodromy action by $u^{\alpha/(d+1)}v^{(d+1-\alpha)/(d+1)}$ if $\alpha\neq 0$, and by $1$ otherwise.  In particular, the operation of taking the Hodge spectrum of a $\muhat$-equivariant motive $X$ and setting $u=v=q^{1/2}=1$ is the same as taking the Euler characteristic of $X$ (forgetting monodromy).  
\medbreak
We deduce that the numerical BPS contribution from the curve $C_d$ is precisely $d$, in agreement with the BPS contribution as defined and calculated in \cite[Thm.1.5]{BrKaLe} in the context of Gromov--Witten theory.  There, the crucial tool is the deformation invariance of Gromov--Witten invariants.  Probably one could derive the numerical version of our result on the contribution of the curve $C_d$ by deforming it to $d$ $(-1,-1)$ curves as in \cite{BrKaLe} and interpreting the calculation of \cite{Dominos} (see \cite[Thm.2.7.2]{conifold}) as stating that the numerical specialization of the motivic contribution of a $(-1,-1)$-curve is $1$.  On the other hand, again with reference to \cite[Thm.2.7.2]{conifold}, we see that deformation invariance of the BPS contribution fails at the motivic level, and even at the level of the Hodge spectrum, as all of the invariants of [ibid] lie inside the subring $\overline{\KK}(\Sta/\Spec(\Cp))[\LL^{1/2}]\subset \overline{\KK}^{\muu}(\Sta/\Spec(\Cp))[\LL^{1/2}]$ of monodromy-free motives, while our calculation of $\Omega_{\zeta}(\nn)$ has nontrivial monodromy.
\end{rem}

\bibliographystyle{plain}
\bibliography{SPP}

\vfill
\textsc{\small B. Davison: IMJ at Universit\'{e} Paris 7, 16 Rue Clisson, 75013 Paris, France}\\
\textit{\small E-mail address:} \texttt{\small davison@math.jussieu.fr}\\
\\

\textsc{\small S. Meinhardt: Mathematisches Institut, Universit\"at Bonn, Endenicher Allee 60, 53115 Bonn, Germany}\\
\textit{\small E-mail address:} \texttt{\small sven@math.uni-bonn.de}\\
\end{document}